\documentclass[11pt]{article}

\usepackage[margin=1in]{geometry}  
\usepackage[utf8]{inputenc}
\usepackage[colorlinks=true,breaklinks=true,linkcolor=lightgreen,citecolor=lightblue]{hyperref}
\usepackage{graphicx}           
\usepackage{amsmath, accents}       
\usepackage{amsfonts}            
\usepackage{dsfont} 
\usepackage{mathrsfs}
\usepackage{amsthm}    
\usepackage{amssymb}           
\usepackage{soulutf8}
\usepackage{xcolor}
\usepackage{longtable}
\usepackage{float}
\usepackage{algorithm}
\usepackage{algpseudocode}
\usepackage{geometry}
\usepackage{amsfonts,mathtools}
\usepackage{amsxtra}
\usepackage{amstext}
\usepackage{amsbsy}
\usepackage{amscd}
\usepackage{float}
\usepackage{enumitem}
\usepackage{todonotes}
\usepackage{siunitx}
\usepackage{fancyhdr}
\usepackage{graphbox}

\newtheorem{thm}{Theorem}[section]
\newtheorem{lem}[thm]{Lemma}
\newtheorem{defi}[thm]{Definition}

\newtheorem{prop}[thm]{Proposition}

\theoremstyle{remark}

\newcommand{\mb}[1]{\mathbf{#1}} 
\newcommand{\bs}[1]{\boldsymbol{#1}}

\newcommand{\Real}{\mathbb{R}}  
\newcommand{\Div}{\mathrm{Div}\,}

\newcommand{\tr}{\mathrm{tr}}

\newcommand{\Eps}{\mathcal{E}} 
\newcommand{\dx}{\ensuremath{\, \mathrm{d}x}}
\newcommand{\ds}{\ensuremath{\, \mathrm{d}s}}

\newcommand{\Hdiv}{\mathbb{H}_{\mathrm{div}}^{p'}(\Omega)}
\newcommand{\Lskew}{\mathbb{L}_{\mathrm{skew}}^{p}(\Omega)}

\DeclarePairedDelimiter\norm{\lVert}{\rVert}
\DeclarePairedDelimiter\abs{\lvert}{\rvert}
\DeclarePairedDelimiter\angles{\langle}{\rangle}
\newcommand{\Norm}{\norm{\cdot}}

\definecolor{lightblue}{rgb}{0.22,0.45,0.70}
\definecolor{lightgreen}{rgb}{0.22,0.55,0.20}
\newcommand{\cblue}[1]{\textcolor{black}{#1}}

\begin{document}

\nocite{*}
\title{\textbf{A Dual-Mixed Approximation for a Huber Regularization of Generalized $p$-Stokes Viscoplastic Flow Problems.}}



\author{
Sergio~Gonz\'alez-Andrade \\
                \footnotesize Research Center on Mathematical Modeling (MODEMAT) and \\\footnotesize Departamento de Matem\'atica -  Escuela Polit\'ecnica Nacional\\\footnotesize Quito 170413, Ecuador.\\\footnotesize\tt sergio.gonzalez@epn.edu.ec 
        \and
 Paul E. M\'endez \\
                \footnotesize Research Center on Mathematical Modeling (MODEMAT) 
                -  Escuela Polit\'ecnica Nacional\\\footnotesize Quito 170413, Ecuador.\\\footnotesize\tt paul.mendez01@epn.edu.ec
}

\date{\today}
\maketitle

\begin{abstract}
In this paper, we propose a dual-mixed formulation for stationary viscoplastic flows with yield, such as the Bingham or the Herschel-Bulkley flow. The approach is based on a Huber regularization of the viscosity term and a two-fold saddle point nonlinear operator equation for the resulting weak formulation. We provide the uniqueness of solutions for the continuous formulation and propose a discrete scheme based on Arnold-Falk-Winther finite elements. The discretization scheme yields a system of slantly differentiable nonlinear equations, for which a semismooth Newton algorithm is proposed and implemented. Local superlinear convergence of the method is also proved. Finally, we perform several numerical experiments in two and three dimensions to investigate the behavior and efficiency of the method.

  \vskip .2in

\vspace{0.2cm}
\end{abstract}

	


\section{Introduction}
\subsection{Scope} 
This paper is devoted to the numerical solution of stationary flows of non-Newtonian fluids where the material exhibits yielded and unyielded regions, such as the Bingham and the Herschel-Bulkley flow. The Navier-Stokes equations model the movement of isothermal fluids described by the mass and momentum conservation equations. These models require a constitutive equation to define the value of the fluid stresses as a function of the flow dynamics. This constitutive equation is associated with the rheological model. In the case of the so called viscoplastic fluids, the main rheological property is the existence of a yield limit or plasticity threshold. This implies that in regions where the stress is below this threshold, the material behaves as a rigid solid. On the other hand, the material behaves as a (complex) fluid in regions where the stress overpasses the yield limit. The model that started the study of viscoplastic materials was Bingham's plastic model.  However, since several different materials exhibit this particular rheological behaviour, a generalization was proposed by W. Herschel and R. Bulkley in \cite{herschel1926}.

The nowadays classic Herschel-Bulkley and Carreau models can be understood as power-law type models with plasticity. This perspective implies the analysis of a model involving a $p$-Laplacian differential operator. These models have been used to simulate a wide range of materials from polymer solutions or blood, which are shear-thinning fluids, to sand soaked with water or thick suspensions of clays, which exhibit a shear-thickening behavior. The versatility of these models comes from the fact that the $p$-Laplacian structure allows to consider explicitly the parameter known as flow index in its formulation. This constant measures the degree to which the fluid is shear-thinning, when $1 < p < 2$ or shear-thickening, if $p > 2$.  Further, if $p=2$, the Herschel-Bulkley model gives the Bingham model, which is a benchmark material. A more gradual change of the material structure, along with the deformation rate, can be achieved with the Papanastasiou’s modification by introducing a sigmoidal term \cite{Kim}. Another two-parameter model, which implicitly shows a shear-thinning/thickening property is the Casson model. Compared with the Herschel–Bulkley model, the viscosity of the Casson model converges to a nonzero asymptotic value and the total plastic viscosity depends of the yield stress \cite{Glowinski2011}.

In this paper, the steady viscoplastic flow with yield is studied by employing a dual-mixed formulation along with a Huber based local regularization. Concerning the numerical approximation, we propose a discrete scheme based on Arnold-Falk-Winther finite elements. The discretization scheme yields a system of Newton differentiable nonlinear equations, for which a semismooth Newton algorithm is used.

Although there are several contributions focused on the numerical and theoretical analysis of viscoplastic materials, to the best of our knowledge, there is scarce work combining the numerical analysis of a mixed finite element approach with computational results in both 2D and 3D flows. This is precisely the gap that we are looking to fulfill with the present work.  The Huber regularization has proved to be an efficient approach for the numerical solution of viscoplastic models, while the novel technique of the dual-dual formulation is an efficient way to handle generalized Stokes and $p$-Stokes problems, such as the Herschel-Bulkley, the Carreau and the Casson models. Further, the semismooth Newton method combined with the Huber approach represents an efficient way to obtain reliable results with superlinear convergence behavior.  Finally, the versatility of the combined approach proposed in this paper allows us to obtain reliable and accurate simulations in both 2D and 3D flows.

\subsection{Related work} 
In \cite{Kim} the authors present and analyze a flow rate based formulation for laminar flow through a slit between parallel plates.  The flow
rate and the velocity are expressed in integral forms and the method is applied for the Bingham, Herschel–Bulkley, Herschel–Bulkley with slip, Bingham–Papanastasiou and Windhab models.

A finite element approximation of a non-Newtonian flow, where the viscosity obeys a general law including the Carreau or power law is proposed in \cite{Barret}. The authors analyze the $\mathbb{P}1$ flux – $\mathbb{P}0$ scheme of Fortin and derive energy type error bounds for the velocity and pressure. Baranger and Najib \cite{Baranger} proved abstract error estimates for the approximation of the velocity and the pressure by a mixed FEM of quasi-Newtonian flows whose viscosity obeys the power-law or the Carreau law. These estimates can be applied to most finite elements used for the solution of Stokes's problem.

The so-called \emph{dual-dual} variational formulations were first suggested  in \cite{gatica1996,gatica1997}. The method is based on the introduction of the gradient (in potential theory and heat conduction) or the strain tensor (in elasticity) as an additional unknown, which yields two-fold saddle point operator equations as the corresponding weak formulations. The extension of the method using stable mixed finite elements from linear elasticity to a nonlinear boundary value problem in plane hyperelasticity, was introduced in \cite{gatica2021}. In that work, the authors also provide an \textit{a-posteriori} error analysis for the two-fold saddle point variational formulations.

In \cite{ervin2008}, the formulation described in \cite{gatica2004} is recast for the numerical approximation of a nonlinear generalized Stokes problem. In doing so, the authors extend the analysis of Gatica et al. \cite{gatica2004}, to more general Sobolev spaces. Between the advantages of this approach, the authors point out: more flexibility in choosing the approximating finite element space for $\mb{u}$, Dirichlet boundary conditions for $\mb{u}$ become natural boundary conditions and are easily incorporated into the variational formulations, and the method avoids the assumption of expressing $\nabla\mb{u}$ as a function of $\bs{\sigma}$. All of this at the cost of introducing additional unknowns.   Note that even if we share many analytical results with Ervin et al. \cite{ervin2008} work, our model considers an explicit dependency of the viscosity on the fluid deformation tensor, and the presence of yielded and unyielded regions in the material. This in turn implies an extra layer of complexity to the numerical scheme, including the presence of a non-differentiable term. We introduce a regularization and a Newton semismooth \cblue{approach} to deal with such complexities.

\cblue{In the same line, a mixed finite element analysis of a three field formulation (velocity, pressure and non-linear stress tensor), of the non-linear Stokes problem was presented in \cite{Manouzi2001}. In that work, authors also derive stable and optimal error estimates. A closer work by Farhloul and Zine \cite{Farhloul2017}, analyzed a dual mixed finite element method for a quasi-Newtonian flow obeying the Carreau or power law (however, without yield). Well-posedness and stability of the discrete formulation is presented, together with  optimal error estimates. The mixed-hybrid finite element method has also been used to study similar flow problems. For a review of these formulations including mixed finite element methods for quasi-Newtonian flows obeying the power law, we refer the reader to the work of Farhloul and Fortin \cite{Farhloul2002}.}

Other contributions in the context of viscoplastic fluids include the models devoted to the study of Bingham fluids. Existence and regularity results for the d-dimensional Bingham fluid flow problem with Dirichlet boundary conditions are presented in \cite{Duvaut1976, Fuchs1998}. Dean \cblue{et al.}  reviewed various results and methods concerning the numerical simulation of Bingham visco-plastic flow in \cite{Deanetal}. In \cite{Apos}, the authors propose a mixed formulation of the Bingham fluid flow problem using a Bercovier–Engelman regularization but similar to our work, they also introduce an auxiliary symmetric tensor to enhance the numerical properties of the regularized term. Furthermore, it is analyzed and demonstrated numerically that the formulation with the auxiliary tensor leads to much better convergence rates than the original formulation. An augmented Lagrangian method together with an incompressible finite element approximation for the flow in a driven cavity is studied in \cite{Roquet1996}. In \cite{JC2010, JC2012} the authors present a scheme based on semismooth Newton methods together with a finite element method ((cross-grid $\mathbb{P}1)–\mathbb{Q}0$ elements) for the numerical simulation of two-dimensional stationary Bingham fluid flow. Well-posedness and convergence of the regularized problems to the original multiplier system are also verified. Regarding the numerical simulation of viscoplastic fluids, including Herschel-Bulkley fluids, we can cite \cite{Timm}, where the authors propose accelerated first order optimization algorithms for a dual formulation of the studied models.

\subsection{Outline of the paper} 

We have organized the contents of this paper in the following manner. Section~\ref{sec:model} describes the general governing equations, the constitutive relations, and the Huber regularization. Section~\ref{sec:dual-dual} outlines the dual-dual mixed formulation of the system. We also introduce the weak formulation and present the well-posedness analysis of the resulting problem in detail. In Section~\ref{sec:numerical} we introduce the Galerkin discretization and define the fully discrete method, considering a semismooth Newton (SSN) linearization. Moreover, we discuss the well-possedness of the discrete formulation and the superlinear convergence of the SSN method. Section~\ref{sec:results} is devoted to the computational results, including two-dimensional and three-dimensional examples under different conditions. We close the paper with some remarks and discussions given in Section~\ref{sec:concl}.  

\section[Problem Statement and Regularization]{Problem Statement and Regularization} \label{sec:model}

We start by introducing some notation. We denote by $L^p(\Omega)$ and  $W^{r,p}(\Omega)$ the usual Lebesgue and Sobolev spaces  with respective norms $\smash{\Norm_{L^p(\Omega)}}$ and $\smash{\Norm_{W^{r,p}(\Omega)}}$. If $p=2$ we write $H^r(\Omega)$ in place of $\smash{W^{r,p}(\Omega)}$. By~$\mb{L}$ and~$\mathbb{L}$, we denote the corresponding vectorial and tensorial counterparts of the scalar functional space~$L$. Moreover, for any vector field $\mb{v} = (v_i)_{i = 1 , d}$ we set the gradient, symmetric part of the gradient and divergence, as
\begin{equation*}
	\nabla \mb{v} := \left(\frac{\partial v_i}{\partial x_j} \right)_{i,j=1,d}, \, \Eps \mb{v} :=\frac{1}{2} \left(\nabla \mb{v} + (\nabla \mb{v})^T\right)\,\mbox{  and  }\, \nabla \cdot \mb{v} := \sum_{j=1}^d \frac{\partial v_j}{\cblue{\partial} x_j},
\end{equation*}
respectively. In addition, for any tensor field $\mb{T} = (\mb{T}_{ij})_{i,j = 1 , d}$, we let $\Div \mb{T}$ be the divergence operator acting along the rows of $\mb{T}$. The Frobenius scalar product in $\mathbb{R}^{d\times d}$ and its associated norm are defined by
\[
(\mb{A} \,:\, \mb{B}) \coloneqq \tr(AB^T)\, \mbox{  and  }\, |\mb{A}| \coloneqq \sqrt{(\mb{A}\,:\,\mb{A})},
\]
respectively. Let $\Omega$ be an open and bounded domain in $\mathbb{R}^{d}$, $d=2,3$ with Lipschitz boundary $\partial \Omega$. This work is concerned with the numerical resolution of the following system of equations
\begin{equation}\label{Problem}\tag{$\mathcal{P}$}
\begin{array}{ccl}
\,\cblue{-\Div(\nu\left(\abs{\Eps\mathbf{u}}\right)\Eps\mb{u})-\Div \bs{\sigma}_s}+ \nabla \phi =\mathbf{f},&\mbox{in $\Omega$}\vspace{0.2cm}\\\nabla\cdot \mathbf{u} =0,&\mbox{in $\Omega$}\vspace{0.2cm}\\\left. \cblue{\begin{array}{lll}
	\bs{\sigma}_s =\tau_s\frac{\mathcal{E}\mathbf{u}}{|\mathcal{E}\mathbf{u}|}, & \text{if }\abs{\Eps\bs{u}} \neq 0, \vspace{0.2cm}\\
	|\bs{\sigma}_s |\leq \tau_s, & \text{if }\abs{\Eps\bs{u}} = 0,
\end{array}}\right\},&\mbox{in $\Omega$}\vspace{0.2cm}\\\mathbf{u}=\cblue{\mb{u}^D},&\mbox{on $\partial\Omega$,}
\end{array}
\end{equation}
where \cblue{$\mb{u}$ is the velocity, $\phi$ the pressure}, \cblue{$\bs{\sigma}_s$ the deviatoric stress tensor}, $\tau_s>0$ stands for the plasticity threshold (yield stress), $\mathbf{f}\in \mathbf{L}^{p'}(\Omega)$ \cblue{and $\mb{u}^D$ must satisfy the compability condition $\int_{\partial\Omega}\mb{u}^D\cdot\mb{n}\ds=0$}. Here \cblue{we assume a non-linear shear stress rate dependent viscosity with yield stress constitutive equation}, where $\nu(\cdot)$ is a nonlinear \cblue{function whose exact form depends on the particular rheological model}. This nonlinear \cblue{function} is usually given by
\begin{equation} \label{eq:Phi}
\Real^{+}\ni t\mapsto \nu(t):=\frac{\widetilde{\nu}\,'(t)}{t},
\end{equation}
where $\widetilde{\nu}:\Real^+\rightarrow\Real^+$ is a suitable $N$-function.  \cblue{Further, it is possible to show that (see \cite[Sec. 2.2.1]{KreuzPhD})
\[
\nu\left(\abs{\mathbf{T}}\right)\mb{T}= \widetilde{\nu}\,'(|\mathbf{T}|)\frac{\mb{T}}{|\mb{T}|},\,\,\mbox{for any tensor $\mb{T} = (\mb{T}_{ij})_{i,j = 1 , d}$}.
\]} 
Several models for viscoplastic materials fit this formulation, as we describe next:
\begin{itemize}
\item The Herschel-Bulkley model is given by the power law function, as follows
\[
\widetilde{\nu}(t):= \frac{\mu}{p}t^p,\,\mbox{for $1<p<\infty$}\,\,\mbox{  and  }\,\, \nu\left(\abs{\mathbf{T}}\right)\mb{T}=\mu |\mathbf{T}|^{p-2}\mathbf{T}.
\]
\item The Carreau model with yield (see \cite{Kim}) is given by
\[
\widetilde{\nu}(t):=\mu \int_0^t  (1+s^2)^{\frac{p-2}{2}}s\,ds,\,\mbox{for $1<p<\infty$}\,\,\mbox{  and  }\,\,  \nu\left(\abs{\mathbf{T}}\right)\mb{T}=\mu \left(1+|\mathbf{T}|^2\right)^{\frac{p-2}{2}}\mathbf{T}.
\]
\item The Casson model is given by
\[
\widetilde{\nu}(t):=\mu\frac{t^2}{2}+\frac{4}{3}\tau_s^{1/2} t^{3/2}\,\,\mbox{  and  }\,\,  \nu\left(\abs{\mathbf{T}}\right)\mb{T}=\mu \mathbf{T} + 2\left(\frac{\tau_2}{|\mathbf{T}|}\right)^{1/2}\mathbf{T}.
\]
\end{itemize}
In all the cases $\mu>0$ \cblue{is a model constant (i.e. constant plastic viscosity in the Bingham model) \cite{Papanastasiou1997}}. Note that the classical Bingham model corresponds to a particular case of the Herschel-Bulkley model, for $p=2$. Furthermore, other common models are described in \cite{ervin2008,KreuzPhD, Farhloul2017} and references therein. 

Also, we recall that the exponent $p$ in the following sections is a general representation of the suitable exponent according to the selected model.

The main characteristic of the viscoplastic fluids, described by the given models, is the presence of yielded and unyielded regions in the material. \cblue{That} is, regions where the material behaves as a non-Newtonian fluid (yielded regions) and regions where the material moves as a rigid solid (unyielded regions). Theoretically, the yielded and unyielded regions are given by the following sets
\[
\mathcal{A}(\mathbf{u}):=\{x\in \Omega\,:\, |\mathcal{E}\mathbf{u}(x)|\neq 0\}\quad\mbox{and}\quad\mathcal{I}(\mathbf{u}):=\{x\in \Omega\,:\, |\mathcal{E}\mathbf{u}(x)|= 0\},
\]
respectively.  As one can appreciate,  in the unyielded regions the material, seen as a rigid solid ($|\mathcal{E}(\mathbf{u})|=0$), is expected to have an infinite viscosity.  These nonlinear behavior, and the fact that, in general, one does not have an a priori knowledge about the yielded and unyielded regions in the domain, \cblue{represents} the main issue in the analysis of these models, making the identification of these regions a key issue in the analysis and numerical resolution of these materials \cite{Apos}.

Furthermore,  \eqref{Problem} is an ill-posed problem, since, \cblue{in general,} it does not have a unique solution. Therefore, we can face instabilities in several numerical schemes. Thus, in order to have a well posed problem and computable approximations of the yielded and unyielded regions,  a regularization approach \cblue{to the non-linear viscosity with yield stress model} is usually proposed.  However, the kind of regularization process is crucial in order to have reliable approximations for these materials.

We propose a local regularization, based on the Huber approach, for the system \eqref{Problem}, which yields the following system of PDEs
\cblue{\begin{equation}\label{Probreg}\tag{$\mathcal{RP}$}
\left\{\begin{array}{ccl}
-\Div \bs{\sigma}_{s,\gamma}+ \nabla \phi =\mathbf{f},&\mbox{in $\Omega$}\vspace{0.2cm}\\\nabla\cdot \mathbf{u} =0,&\mbox{in $\Omega$}\vspace{0.2cm}\\\bs{\sigma}_{s,\gamma}=\nu(|\mathcal{E}\mathbf{u}|)\mathcal{E}\mathbf{u}+\tau_s\gamma\frac{\mathcal{E}\mathbf{u}}{|\mathcal{E}\mathbf{u}|_\gamma}&\mbox{in $\Omega$}\vspace{0.2cm}\\ \mathbf{u}=\cblue{\mb{u}^D},&\mbox{on $\partial\Omega$}
\end{array}\right.
\end{equation}}Here for a parameter $\gamma > 0$, $|\mathcal{E}\mathbf{u}|_\gamma:=\max(\tau_s, \gamma|\mathcal{E}\mathbf{u}|)$.  In the next section, we will prove that the Huber regularized system \eqref{Probreg}, in a dual mixed formulation, is a well posed problem with a unique solution, for each $\gamma>0$. 

The main idea behind this regularization is to change the plastic component of the shear dependent viscosity $\tau_s \frac{\mathcal{E}\mathbf{u}}{|\mathcal{E}\mathbf{u}|}$ by the following expression $\tau_s\gamma\frac{\mathcal{E}\mathbf{u}}{|\mathcal{E}\mathbf{u}|_\gamma}$, which depends on the parameter $\gamma$. By doing this, we obtain the following approximations for the yielded and unyielded regions sets $\mathcal{A}$ and $\mathcal{I}$
\[
\mathcal{A}_\gamma(\mathbf{u}):=\{x\in \Omega\,:\, \gamma|\mathcal{E}\mathbf{u}|\geq \tau_s\}\quad\mbox{and}\quad \mathcal{I}_\gamma(\mathbf{u}):=\{x\in \Omega\,:\, \gamma|\mathcal{E}\mathbf{u}|< \tau_s\},
\]
respectively.  This approach allows us, for a sufficiently big $\gamma$,  to recover the actual non-Newtonian behavior of the fluid in $\mathcal{A}_\gamma$, while the theoretical infinite viscosity of the material in the unyielded regions is correctly approximated by a ``sufficiently big'' viscosity in $\mathcal{I}_\gamma$.  Consequently, we can assure that the Huber approach allows to \cblue{perform} reliable numerical simulations of the fluids, preserving the qualitative properties of the model in most of the domain, since the effect of the regularization is focused only on a small neighbourhood of the unyielded regions.  

Because of the fact that the viscosity is fixed on $\mathcal{A}_\gamma$ and $\mathcal{I}_\gamma$, for $\gamma>0$, this approach is also known as a bi-viscosity regularization (see \cite{bever}). 

\section{Dual Mixed approximation} \label{sec:dual-dual}
In this section based on the works of \cite{ervin2008} and \cite{Barrientos2002} we devise a model for stationary flows considering a dependency of the viscosity on the fluid deformation tensor, and the presence of yielded and unyielded regions in the material. In order to obtain the dual-mixed formulation, we proceed as in \cite{almonacid2020} and set the strain rate tensor as an auxiliary unknown
\[
\bs{\theta}= \Eps \mb{u} =  \nabla \mb{u} - \bs{\gamma}(\mb{u}),
\]
where $\bs{\gamma}(\mb{u}) = \frac{1}{2}(\nabla \mb{u} - (\nabla \mb{u})^T)$  is the skew-symmetric part of the velocity gradient tensor. Further, for any $p'>1$ and $p$ such that $1/p + 1/p' = 1$, we define
\[
\Hdiv \coloneqq \{\bs{\tau} \in \mathbb{L}^{p'}(\Omega) \,:\, \Div \bs{\tau} \in \mb{L}^{p'}(\Omega) \}\,\,\mbox{  and  }\,\,\Lskew \coloneqq \{\bs{\gamma} \in \mathbb{L}^{p}(\Omega) \,:\, \bs{\gamma}+\bs{\gamma}^T = 0 \}.
\]
The introduction of the new unknown $\bs{\theta}$ leads us to the following representation of the total stress tensor
\[
\bs{\sigma} = \nu(|\bs{\theta}|)\bs{\theta}+ \gamma \tau_{s} \frac{\bs{\theta}}{|\bs{\theta}|_{\gamma}}  - \phi \mathbb{I}.
\]
Here $\mathbb{I}$ is the identity tensor of dimension $d\times d$. Hence, by introducing $\hat{\mb{u}} := \bs{\gamma}(\mb{u}) \in \Lskew $ as a new variable from \eqref{Problem}, we arrive at the following system: 

\begin{equation} \label{eq:system_mix_dual}
	\begin{array}{rllc}
	\bs{\theta}&= \nabla \mb{u} - \hat{\mb{u}} = \Eps \mb{u},  &\mathrm{\text{in }} \Omega, &\\
	\bs{\sigma} &= \nu(|\bs{\theta}|)\bs{\theta} + \gamma \tau_{s} \frac{\bs{\theta}}{|\bs{\theta}|_{\gamma}} - \phi \mathbb{I}, &\mathrm{\text{in }} \Omega, &\\
	-\Div(\bs{\sigma}) &= \mb{f}, &\mathrm{\text{in }} \Omega, &\\
	\tr(\bs{\theta}) & =0, &\mathrm{\text{in }} \Omega, &\\
	\mb{u} &= \cblue{\mb{u}^D}, &\mathrm{\text{on }} \partial\Omega.& 
	\end{array}
\end{equation}

Note that the equation $\tr(\bs{\theta})=0$ stands for the incompressibility condition. 

\subsection{Weak formulation}
Given $\mb{f} \in \mb{L}^{p'}(\Omega)$, testing each equation in problem \eqref{eq:system_mix_dual} against suitable functions and integrating by parts whenever adequate, gives: find $(\bs{\theta}, \bs{\sigma} ,\phi, \mb{u},\hat{\mb{u}}, \lambda) \in \mathbb{L}^{p}(\Omega)\times \Hdiv \times L^{p'}(\Omega) \times \mb{L}^{p}(\Omega)\times \Lskew \times \mathbb{R}$ such that
\begin{equation} \label{eq:var_mod}
	\begin{aligned}
	-\int_{\Omega} \bs{\theta} \,:\,\bs{\tau} \dx - \int_{\Omega} \psi \tr(\bs{\theta}) \dx - \int_{\Omega} \mb{u} \cdot \Div \bs{\tau} \dx
	 \\  - \int_{\Omega} \hat{\mb{u}} : \bs{\tau} \dx + \lambda \int_{\Omega} \tr(\bs{\tau}) \dx &= \cblue{-\int (\bs{\tau}\cdot\bs{n})\cdot\mb{u}^D\ds}\\
	\int_{\Omega} \nu(|\bs{\theta}|)(\bs{\theta} : \bs{\xi}) \dx + \gamma \tau_{s} \int_{\Omega} \frac{1}{|\bs{\theta}|} (\bs{\theta} : \bs{\xi}) \dx - \int_{\Omega} \bs{\sigma} : \bs{\xi} \\
	- \int_{\Omega} \phi \tr(\bs{\xi}) \dx &=0 \\
	-\int_{\Omega} \mb{v} \cdot \Div \bs{\sigma} \dx - \int_{\Omega} \hat{\mb{v}}: \bs{\sigma} \dx + \eta \int_{\Omega} \tr(\bs{\sigma}) \dx &= \int_{\Omega} \mb{v} \cdot \mb{f} \dx
	\end{aligned}
\end{equation}
for all $(\bs{\tau},\psi) \in \Hdiv \times L^{p'}(\Omega)$, $\bs{\xi} \in \mathbb{L}^{p}(\Omega)$ and $(\mb{v},\hat{\mb{v}},\eta) \in  \mb{L}^{p}(\Omega)\times \Lskew \times \mathbb{R}$.

\cblue{Here the Dirichlet boundary condition is imposed as a natural boundary condition. However other common boundary conditions can also be studied with a similar formulation. For example, mixed conditions involving the normal stress can be added as essential boundary conditions imposed in the space of tensor $\bs{\sigma}$. The following analysis still carries on for this setting, since an appropriate inf-sup condition still holds for that space (see \cite[Lemma 4.3]{Barrientos2002}).}

Similar to \cite{almonacid2020, arnold2007} the symmetry condition for the space of matrix fields $\Hdiv$ is enforced weakly through the introduction of the Lagrange multiplier $\hat{\mb{v}}$. Also, the condition of zero-averaged fluid pressure (translated in terms of tensor traces and required for uniqueness of the solution) is imposed through a real Lagrange multiplier $\eta$.

Now, the forms $a_{\gamma}$, $b$ and $c$ are defined, for all $\bs{\theta}, \bs{\xi}  \in \mathbb{L}^{p}(\Omega)$, $\bs{\tau} \in  \Hdiv$, $\psi \in L^{p'}(\Omega)$, $\phi \in  L^p(\Omega)$, $\eta \in \mathbb{R}$, $\hat{\mb{v}} \in \Lskew$ and $\mb{v} \in \mb{L}^{p}(\Omega)$, as follows:
\begin{align*}
a_{\gamma}(\bs{\theta}, \bs{\xi}) &\coloneqq \int_{\Omega} \nu(|\bs{\theta}|)(\bs{\theta} : \bs{\xi}) \dx + \gamma \tau_{s} \int_{\Omega} \frac{1}{|\bs{\theta}|_{\gamma}} (\bs{\theta} : \bs{\xi}) \dx, \\
b(\bs{\theta},(\bs{\tau},\psi)) &\coloneqq - \int_{\Omega} \bs{\theta} : \bs{\tau} \dx - \int_{\Omega}\psi \tr(\bs{\theta}) \dx,  \\
c((\bs{\sigma},\phi),(\mb{v},\hat{\mb{v}},\eta)) &\coloneqq -\int_{\Omega} \mb{v} \cdot \Div \bs{\sigma} \dx + \eta \int_{\Omega} \tr(\bs{\sigma}) \dx - \int_{\Omega} \hat{\mb{v}} : \bs{\sigma} \dx.
\end{align*}
In this way, system \eqref{eq:var_mod} can be written as a twofold saddle point problem of the form:
\begin{subequations} \label{eq:saddle-point-problem}
\begin{align}
a_{\gamma}(\bs{\theta},\bs{\xi}) + b^t(\bs{\xi},(\bs{\sigma},\phi)) &= 0 &&\text{for all } \bs{\xi} \in \mathbb{L}^{p}(\Omega) \label{eq:spp_1} \\
b(\bs{\theta},(\bs{\tau},\psi)) + c^t((\bs{\tau},\psi),(\bf{u},\hat{\mb{u}},\lambda)) &=\cblue{-\int (\bs{\tau}\cdot\bs{n})\cdot\mb{u}^D\ds} &&\text{for all }(\bs{\tau}, \psi) \in \Hdiv \times L^{p'}(\Omega) \label{eq:spp_2}\\
c((\bs{\sigma}, \phi),(\mb{v},\hat{\mb{v}},\eta)) &= \int_{\Omega} \mb{f} \cdot \mb{v} && \text{for all } (\mb{v},\eta) \in  \mb{L}^{p}(\Omega)\times \mathbb{R}.
\end{align}
\end{subequations}

\subsection{ Well-posedness of the problem}

In this section, we discuss the existence and uniqueness of a solution to \eqref{eq:saddle-point-problem} with focus on the considered models. The proof of this result is based on the general theory
of saddle point problems, including Banach-Ne\v{c}as-Babu\v{s}ka Theorem, and the classical Babu\v{s}ka-Brezzi theory, both in Banach spaces (see \cite[Theorems 2.6, 2.34]{ern2004}), and monotone operator theory. We first recall some results that will be used to establish the solvability in Theorem \ref{thm:wellp_var_mod}. 
Note that some results are similar to those used to prove the solvability of the problem in \cite{ervin2008}, we still mention the main ones for the sake of completeness.

We begin with the following auxiliary result related \cblue{to} the stability of the regularization. 


\begin{lem}\label{lem:E<gamma}
	Let $\boldsymbol{\theta},\,\boldsymbol{\vartheta}\in \mathbb{L}^p(\Omega)$. Then, the following estimate holds
	\begin{equation}\label{E<gamma}
		|\boldsymbol{\theta}(x)|_\gamma - |\boldsymbol{\vartheta}(x)|_\gamma\leq \gamma |\boldsymbol{\theta}(x) - \boldsymbol{\vartheta}(x)|,\,\mbox{ a.e. in $\Omega$}.
	\end{equation}
\end{lem}
\begin{proof}
Let us start by noticing that, because of the introduction of the auxiliary unknown $\bs{\theta}$, the approximations for the yielded and unyielded regions are rewritten as
\[
\mathcal{A}_\gamma(\bs{\theta}):=\{x\in \Omega\,:\, \gamma|\bs{\theta}(x)|\geq \tau_s\}\quad\mbox{and}\quad \mathcal{I}_\gamma(\bs{\theta}):=\{x\in \Omega\,:\, \gamma|\bs{\theta}(x)|< \tau_s\}.
\]
Next, we analyze the behaviour of \eqref{E<gamma} in the following sets $\mathcal{A}_\gamma(\bs{\theta})\cap \mathcal{A}_\gamma(\bs{\vartheta})$, $\mathcal{A}_\gamma(\bs{\theta})\cap \mathcal{I}_\gamma(\bs{\vartheta})$, $\mathcal{I}_\gamma(\bs{\theta})\cap \mathcal{A}_\gamma(\bs{\vartheta})$ and $\mathcal{I}_\gamma(\bs{\theta})\cap \mathcal{I}_\gamma(\bs{\vartheta})$. \\
	
	\noindent \underline{On $\mathcal{A}_\gamma(\bs{\theta})\cap \mathcal{A}_\gamma(\bs{\vartheta})$}: Here, we have that
	\[
	|\boldsymbol{\theta}(x)|_\gamma - |\boldsymbol{\vartheta}(x)|_\gamma= \gamma(|\boldsymbol{\theta}(x)| - |\boldsymbol{\vartheta}(x)|)\leq \gamma  |\boldsymbol{\theta}(x) -\boldsymbol{\vartheta}(x)|.
	\]
	
	\noindent \underline{On $\mathcal{A}_\gamma(\bs{\theta})\cap \mathcal{I}_\gamma(\bs{\vartheta})$,}: Here, it holds that
	\[
	|\boldsymbol{\theta}(x)|_\gamma - |\boldsymbol{\vartheta}(x)|_\gamma= \gamma|\boldsymbol{\theta}(x)| - \tau_s < \gamma|\boldsymbol{\theta}(x)|- \gamma|\boldsymbol{\vartheta}(x)|\leq  \gamma |\boldsymbol{\theta}(x) -\boldsymbol{\vartheta}(x)|.
	\]
	
	\noindent \underline{On $\mathcal{I}_\gamma(\bs{\theta})\cap \mathcal{A}_\gamma(\bs{\vartheta})$}: Here, we know that
	\[
	|\boldsymbol{\theta}(x)|_\gamma - |\boldsymbol{\vartheta}(x)|_\gamma= \tau_s - \gamma|\boldsymbol{\vartheta}(x)| \leq \tau_s-\tau_s=0\leq  \gamma |\boldsymbol{\theta}(x) -\boldsymbol{\vartheta}(x)|.
	\]
	
	\noindent \underline{On $\mathcal{I}_\gamma(\bs{\theta})\cap \mathcal{I}_\gamma(\bs{\vartheta})$}: Here, we obtain the following
	\[
	|\boldsymbol{\theta}(x)|_\gamma - |\boldsymbol{\vartheta}(x)|_\gamma= \tau_s-\tau_s=0\leq  \gamma |\boldsymbol{\theta}(x) - \boldsymbol{\vartheta}|.
	\]
Thus, since the considered sets provide a disjoint partitioning of $\Omega$, the four estimates above imply \eqref{E<gamma}.
\end{proof}

Now, we can \cblue{establish} the following theorem concerning the coerciveness, strict monotonicity, and hemicontinuity of the operator $\mathbf{A}_{\gamma}$ associated with the form $a_{\gamma}$.

\begin{lem}\label{lem:opAmon}
	Let $1<p<\infty$. The operator $\mathbf{A}_\gamma:\mathbb{L}^{p}(\Omega)\rightarrow [\mathbb{L}^{p'}(\Omega)]^*$, defined as
	\begin{equation*}
	\langle \mathbf{A}_\gamma(\bs{\theta})\,,\,\bs{\xi}\rangle_{[\mathbb{L}^{p'}(\Omega)]^*,\mathbb{L}^{p}(\Omega)}:= a_{\gamma}(\bs{\theta},\bs{\xi}), \mbox{  for all $\bs{\xi}\in \mathbb{L}^{p}(\Omega)$},
	\end{equation*}
	is coercive, hemicontinuous and strictly monotone.
\end{lem}
\begin{proof}
We proceed by first proving the coercivity of $\mathbf{A}_\gamma$.
	\[
	\begin{array}{lll}
	\angles{\mathbf{A}_\gamma(\bs{\theta})\,,\,\bs{\theta}}_{[\mathbb{L}^{p'}(\Omega)]^*,\mathbb{L}^{p}(\Omega)} &=&\mu \int_\Omega \nu(|\bs{\theta}|)|\bs{\theta}|^2\,dx + \gamma\tau_s \int_\Omega \frac{|\bs{\theta}|^2}{|\bs{\theta}|_\gamma}\,dx\vspace{0.2cm}\\ &\geq&  \int_\Omega \nu(|\bs{\theta}|) |\bs{\theta}|^2\,dx.
	\end{array}
	\]
This last inequality implies the coercivity of the operator in the proposed models and in similar models given by $N$-functions. In fact, note that 
\begin{enumerate}
\item In the case of the Herschel-Bulkley model, we have that
\[
\int_\Omega \nu(|\bs{\theta}|) |\bs{\theta}|^2\,dx=\mu\int_\Omega |\bs{\theta}|^{p-2}|\bs{\theta}|^2\,dx=\mu\int_\Omega |\bs{\theta}|^p\,dx= \mu\|\bs{\theta}\|_{\mathbb{L}^p}^p.
\]
\item For the Carreau model, we know,  from \cite[Prop. 3.1]{Baranger} and \cite[Lem. 3.1]{Barret2}, that \cblue{there exists} positive constants $C_1$ and $C_2$ such that
\[
\int_\Omega \nu(|\bs{\theta}|) |\bs{\theta}|^2\,dx\geq \left\{\begin{array}{lll}
C_1\int_\Omega (1+|\bs{\theta}|^{p-2})|\bs{\theta}|^2\,dx,&\mbox{if $1<p<2$,}\vspace{0.2cm}\\C_2\int_\Omega |\bs{\theta}|^p\,dx,&\mbox{if $p\geq 2$}.
\end{array}\right.
\]
\item In the case of Casson model, we have that
\[
\int_\Omega \nu(|\bs{\theta}|) |\bs{\theta}|^2\,dx = \int_\Omega \left[\mu + 2\left(\frac{\tau_s}{|\bs{\theta}|}\right)^{1/2}\right]|\bs{\theta}|^2\,dx = \mu\int_\Omega |\bs{\theta}|^2\,dx + 2\tau^{1/2}\int_\Omega|\bs{\theta}|^{3/2}\,dx\geq \mu\|\bs{\theta}\|^2_{\mathbb{L}^2}.
\]
\end{enumerate}
Thanks to the estimators above, we can conclude that for the three models in study, there \cblue{exists} a positive constant $C$, depending on $p$, such that
	\[
	\dfrac{\langle\mathbf{A}_\gamma(\bs{\theta})\,,\,\bs{\theta}\rangle_{[\mathbb{L}^{p'}(\Omega)]^*,\mathbb{L}^{p}(\Omega)}}{\|\bs{\theta}\|^p_{\mathbb{L}^{p}(\Omega)}}\geq C\|\bs{\theta}\|^{p-1}_{\mathbb{L}^{p}(\Omega)},
	\]
	noticing that, in the Casson model, $p=2$.  Finally,  since $p>1$ in all the cases,  it follows that
	\begin{equation*}
	\dfrac{\langle \mathbf{A}_\gamma(\bs{\theta})\,,\,\bs{\theta}\rangle_{[\mathbb{L}^{p'}(\Omega)]^*,\mathbb{L}^{p}(\Omega}}{\|\bs{\theta}\|^p_{\mathbb{L}^{p}(\Omega)}}\rightarrow \infty\,\mbox{ as }\, \|\bs{\theta}\|_{\mathbb{L}^{p}(\Omega)}\rightarrow\infty.
	\end{equation*}
	
Further, we prove that $\mathbf{A}_\gamma$ is hemicontinuous, \textit{i.e.}, we prove that the function
	\begin{equation*}
	\begin{array}{lll}
	\Real\ni t\mapsto \langle \mathbf{A}_\gamma(\bs{\theta}+t\bs{\tau})\,,\,\bs{\psi}\rangle_{[\mathbb{L}^{p'}(\Omega)]^*,\mathbb{L}^{p}(\Omega)} = \int_\Omega \nu(|\bs{\theta}+ t\bs{\upsilon}|)[(\bs{\theta}+t\bs{\upsilon}):\bs{\psi}]\,dx \vspace{0.2cm}\\\hspace{5cm}+ \gamma\tau_s\int_\Omega \frac{1}{|\bs{\theta}+ t\bs{\upsilon}|_\gamma}[(\bs{\theta}+ t\bs{\upsilon}):\bs{\psi}]\,dx, \quad \forall \bs{\theta}, \bs{\upsilon}, \bs{\psi}\in \mathbb{L}^{p}(\Omega),
	\end{array}
	\end{equation*}
	is continuous. In this aim, let $\{t_n\}\subset(0,1)$ be a sequence convergent to $t\in[0,1]$.  Next in \cite[Lem. 48]{KreuzPhD} is proved that the  form $\int_\Omega \nu(|\bs{\theta}+ t\bs{\upsilon}|)[(\bs{\theta}+t\bs{\upsilon}):\bs{\psi}]\,dx $ is sequentially continuous, which implies that
	\begin{equation}\label{hemi1}
	\lim_{n\rightarrow\infty}\int_\Omega \nu(|\bs{\theta}+ t_n\bs{\upsilon}|)[(\bs{\theta}+t_n\bs{\upsilon}):\bs{\psi}]\,dx =\int_\Omega \nu(|\bs{\theta}+ t\bs{\upsilon}|)[(\bs{\theta}+t\bs{\upsilon}):\bs{\psi}]\,dx. 
	\end{equation}
Let us now focus on the term involving the Huber regularization $|\cdot|_\gamma$.  First, we recall that $|\cdot|_\gamma$ is a continuous function \cblue{in $\mathbb{L}^p(\Omega)$} and that $0<\tau_s\leq |\bs{\theta}+t\bs{\upsilon}|_\gamma$ a.e. in $\Omega$. Therefore, we have, for sufficiently large $n$, that
\[
\gamma\tau_s\frac{1}{|\bs{\theta}+ t_n\bs{\upsilon}|_\gamma}[(\bs{\theta}+ t_n\bs{\upsilon}):\bs{\psi}]\leq \gamma [(\bs{\theta}+ t_n\bs{\upsilon}):\bs{\psi}]\leq \gamma[|\bs{\theta}|+|\bs{\upsilon}|]|\bs{\psi}|.
\]
Since the function $x\mapsto \gamma[|\bs{\theta}(x)|+|\bs{\upsilon}(x)|]|\bs{\psi}(x)|\in\mathbb{L}^1(\Omega)$, the Lebesgue's dominated convergence theorem implies that
\begin{equation}\label{hemi2}
\lim_{n\rightarrow\infty}\gamma\tau_s\int_\Omega \frac{1}{|\bs{\theta}+ t_n\bs{\upsilon}|_\gamma}[(\bs{\theta}+ t_n\bs{\upsilon}):\bs{\psi}]\,dx = \gamma\tau_s\int_\Omega \frac{1}{|\bs{\theta}+ t\bs{\upsilon}|_\gamma}[(\bs{\theta}+ t\bs{\upsilon}):\bs{\psi}]\,dx.
\end{equation}
Clearly, \eqref{hemi1} and \eqref{hemi2} imply the hemicontinuity of  $\mathbf{A}_\gamma$.

Finally, we prove the strict monotonicity of $\mathbf{A}_\gamma$. First, by following \cite[Prop. 43]{KreuzPhD}, we can state that there exists a constant $C>0$ such that
	\begin{eqnarray}
	\langle \mathbf{A}_\gamma(\bs{\theta})- \mathbf{A}_\gamma(\bs{\psi})\,,\,\bs{\theta}-\bs{\psi}\rangle_{[\mathbb{L}^{p'}(\Omega)]^*,\mathbb{L}^{p}(\Omega)} &\geq& C \int_\Omega \cblue{\widetilde{\nu}\,''}(|\bs{\theta}| + |\bs{\psi}|) |\bs{\theta}-\bs{\psi}|^2\,dx\nonumber\vspace{0.2cm}\\ &&+ \gamma\tau_s\int_\Omega \left( \frac{\bs{\theta}}{|\bs{\theta}|_\gamma} -\frac{\bs{\psi}}{|\bs{\psi}|_\gamma}\right) :(\bs{\theta}  - \bs{\psi})\,dx,\label{eq:mon2}
	\end{eqnarray}
where $\cblue{\widetilde{\nu}}$	is the $N$-function defining each model. Let us focus on \eqref{eq:mon2}. Note that this term can be rewritten as
	
	\begin{equation*}
	\begin{array}{lll}
	\gamma\tau_s\int_\Omega \left[\left(\frac{1}{|\bs{\theta}|_\gamma} -\frac{1}{|\bs{\psi}|_\gamma}\right) \bs{\psi} + \frac{1}{|\bs{\theta}|_\gamma}\left(\bs{\theta}-\bs{\psi}\right)\right] :(\bs{\theta}  - \bs{\psi})\,dx\vspace{0.2cm}\\\hspace{3.4cm}= \gamma\tau_s\int_\Omega \left[ \frac{1}{|\bs{\theta}|_\gamma}(\bs{\theta} -\bs{\psi}) + \left(\frac{|\bs{\psi}|_\gamma - |\bs{\theta}|_\gamma}{|\bs{\psi}|_\gamma |\bs{\theta}|_\gamma}\right) \bs{\psi}\right] :(\bs{\theta} - \bs{\psi}) \,dx\vspace{0.2cm}\\\hspace{4.5cm} = \gamma\tau_s\int_\Omega \frac{1}{|\bs{\theta}|_\gamma} \left[ |\bs{\theta} -\bs{\psi}|^2 - \left(\frac{|\bs{\theta}|_\gamma - |\bs{\psi}|_\gamma}{|\bs{\psi}|_\gamma}\right) \bs{\psi} :(\bs{\theta} - \bs{\psi}) \right]\,dx.
	\end{array}
	\end{equation*}
Then, thanks to Lemma \ref{lem:E<gamma}, the Cauchy-Schwarz inequality,  and since $\left|\frac{\bs{\psi}(x)}{|\bs{\psi}(x)|_\gamma}\right|\leq \frac{1}{\gamma}$ a.e. in $\Omega$, we conclude that
	\begin{equation*}
	\begin{array}{lll}
\gamma\tau_s\int_\Omega \left[\left(\frac{1}{|\bs{\theta}|_\gamma} -\frac{1}{|\bs{\psi}|_\gamma}\right) \bs{\psi} + \frac{1}{|\bs{\theta}|_\gamma}\left(\bs{\theta}-\bs{\psi}\right)\right] :(\bs{\theta}  - \bs{\psi})\,dx\vspace{0.2cm}\\\hspace{3.4cm}\geq \gamma\tau_s\int_\Omega \frac{1}{|\bs{\theta}|_\gamma} \left[ |\bs{\theta} -\bs{\psi}|^2 - \gamma |\bs{\theta} - \bs{\psi} |^2 \left|\frac{\bs{\psi}(x)}{|\bs{\psi}(x)|_\gamma}\right|\right]\,dx\geq 0,
	\end{array}
	\end{equation*}
	which yields that
	\begin{equation*}
	\langle \mathbf{A}_\gamma(\bs{\theta})- \mathbf{A}_\gamma(\bs{\psi})\,,\,\bs{\theta}-\bs{\psi}\rangle_{\mathbb{L}^{p'}(\Omega),\mathbb{L}^{p}(\Omega)}\geq C \int_\Omega  \cblue{\widetilde{\nu}\,''}(|\bs{\theta}| + |\bs{\psi}|) |\bs{\theta}-\bs{\psi}|^2,dx.
	\end{equation*}
	Next, since all the \cblue{$N$-functions} associated with the given models satisfy the so-called $\Delta_2$-condition (see \cite[Def. 8]{KreuzPhD}), \textit{i.e}, there exists $M>0$ such that $\widetilde{\nu}(2t)\leq M\,\widetilde{\nu}(t)$, for $t\geq 0$, \cite[Rem. 42]{KreuzPhD} guarantees that 
\begin{equation}\label{Delta2}
\cblue{\widetilde{\nu}\,''}(t)\approx \frac{\cblue{\widetilde{\nu}\,'}(t)}{t}>0,\,\,\mbox{  on $(0,\infty)$},
\end{equation}
which implies that	
\begin{equation*}
	\langle \mathbf{A}_\gamma(\bs{\theta})- \mathbf{A}_\gamma(\bs{\psi})\,,\,\bs{\theta}-\bs{\psi}\rangle_{\mathbb{L}^{p'}(\Omega),\mathbb{L}^{p}(\Omega)}\geq C \int_\Omega  \cblue{\widetilde{\nu}\,''}(|\bs{\theta}| + |\bs{\psi}|) |\bs{\theta}-\bs{\psi}|^2,dx\geq 0,
	\end{equation*}
\textit{i.e.}, the operator $\mathbf{A}_\gamma$ is monotone.
	
	Now, in order to prove strict monotonicity, let us suppose that
	\[
	\langle \mathbf{A}_\gamma(\bs{\theta})- \mathbf{A}_\gamma(\bs{\psi})\,,\,\bs{\theta}-\bs{\psi}\rangle_{[\mathbb{L}^{p'}(\Omega)]^*,\mathbb{L}^{p}}=0.
	\]
	This expression immediately implies that
	\[
	\int_\Omega\cblue{\widetilde{\nu}\,''}(|\bs{\theta}| + |\bs{\psi}|) |\bs{\theta}-\bs{\psi}|^2,dx=0,
	\]
	which yields that $(\bs{\theta}-\bs{\psi})=0$ a.e. in $\Omega$. Therefore, $\bs{\theta}=\bs{\psi}$ in $\mathbb{L}^{p}(\Omega)$.
\end{proof}


We then need some properties of the \cblue{bilinear} forms $b(\cdot,\cdot)$ and $c(\cdot,\cdot)$ respectively. We characterize the kernel $\mb{K_1}$ of the bilinear form $c(\cdot, \cdot)$ as follows
\begin{align*}
\mb{K_1} &\coloneqq \left\{ (\bs{\sigma}, \phi)\in \Hdiv \times L^{p'}(\Omega) \,:\, c((\bs{\sigma}, \phi),(\mb{v},\hat{\mb{v}}, \eta)) = 0 \right\} \\
&\coloneqq \left\{ (\bs{\sigma}, \phi)\in\Hdiv\times L^{p'}(\Omega) \,:\, \Div \bs{\sigma} = 0 \text{ in } \Omega, \int_{\Omega} \tr(\bs{\sigma})\,\dx = 0, \text{ and } \bs{\sigma} = \bs{\sigma}^t \text{ in } \Omega \right\},
\end{align*}
and state the inf-sup condition of the form $b(\cdot,\cdot)$ in kernel $\mb{K_1}$, in the following Lemma.

\begin{lem} \label{lem:infsup-b}
	There exists $C_b > 0$, such that
	\begin{align*}
	\inf_{(\bs{\tau},\psi) \in \mb{K}_1}  \, \sup_{\bs{\theta}\in \mathbb{L}^{p}(\Omega)} \frac{b(\bs{\theta},(\bs{\tau},\psi))}{\norm{\bs{\theta}}_{\mathbb{L}^{p}(\Omega)} \norm{(\bs{\tau},\psi)}_{ \Hdiv\times L^{p'}(\Omega)}} &\geq C_b.
	\end{align*}
\end{lem}
\begin{proof}
	See \cite[Lemma 3.2]{ervin2008}
\end{proof}

We also need the following technical result.
\begin{lem} \label{lem:sup-uDivT}
	Let $\mathbb{H}^{p'}_{(\mathrm{div},0)}(\Omega) \coloneqq \left\{ \bs{\tau} \in \Hdiv \,:\, \int_{\Omega} \tr(\bs{\tau}) \dx =0\right\}$. Then, there exists $C>0$ such that, for any $\mb{u}\in \mb{L}^p(\Omega)$,
	\begin{align*}
	\sup_{\substack{\bs{\tau}_0 \in \mathbb{H}^{p'}_{(\mathrm{div},0)}(\Omega) \\ \bs{\tau}_0 \neq 0}} \frac{-\int_{\Omega} \mb{u}\cdot \Div \bs{\tau}_0 + \hat{\mb{u}} : \bs{\tau_0} \dx }{\norm{\bs{\tau}_0}_{\Hdiv}} \geq \sup_{\substack{\bs{\tau} \in \Hdiv \\ \bs{\tau} \neq 0}} \frac{-\int_{\Omega} \mb{u}\cdot \Div \bs{\tau} + \hat{\mb{u}} : \bs{\tau} \dx }{\norm{\bs{\tau}}_{\Hdiv}}.
	\end{align*}
\end{lem}
\begin{proof}
	The result follows from the same steps of the proof of \cite[Lemma 3.3]{ervin2008}, taking into account that $\hat{\mb{u}}^d = \hat{\mb{u}}$ and $\bs{\tau}^d:\bs{\sigma} = \bs{\tau}^d:\bs{\sigma}^d$.
\end{proof}

We are ready to prove the inf-sup condition for $c(\cdot,\cdot)$.

\begin{lem} \label{lem:inf-sup-c}
	There exists a constant $C_c > 0$ such that 
	\begin{align*}
	\inf_{(\mb{v},\hat{\mb{v}},\eta) \in \mb{L}^{p}(\Omega)\times\Lskew\times \mathbb{R} }  \, \sup_{(\bs{\sigma},\phi)\in \Hdiv \times L^{p}(\Omega)} \frac{c((\bs{\sigma}, \phi),(\mb{v},\eta))}{\norm{(\bs{\sigma},\phi)}_{ \Hdiv \times L^{p}(\Omega)} \norm{(\mb{v},\hat{\mb{v}},\eta)}_{\mb{L}^{p}(\Omega)\times \Lskew\times \mathbb{R}}} &\geq C_c.
	\end{align*}
\end{lem}
\begin{proof}
\cblue{Let} us consider first,
	\begin{align} \label{eq:sup_lambd}
	\sup_{(\bs{\tau},\phi)\in \Hdiv \times L^{p}(\Omega)} \frac{c((\bs{\tau}, \phi),(\mb{u},\hat{\mb{u}},\lambda))}{\norm{(\bs{\tau},\phi)}_{ \Hdiv \times L^{p}(\Omega)}} &\geq \frac{ c(\lambda \mathbb{I},0),(\mb{u},\hat{\mb{u}},\lambda)}{\norm{\lambda \mathbb{I}}_{\Hdiv}} \nonumber \\
	&= \frac{N\lambda^2\abs{\Omega}}{\abs{\lambda} N^{p'/2}\abs{\Omega}^{1/p'}} \nonumber \\
	& \geq C \abs{\lambda}
	\end{align}	
 Now, using Lemma \ref{lem:sup-uDivT}, we have that
	\begin{align} \label{eq:c_part1}
	\sup_{(\bs{\tau},\phi)\in \Hdiv \times L^{p}(\Omega)} \frac{c((\bs{\tau}, \phi),(\mb{u},\hat{\mb{u}},\lambda))}{\norm{(\bs{\tau},\phi)}_{ \Hdiv \times L^{p}(\Omega)}} &\geq \sup_{\bs{\tau}_0 \in \mathbb{H}^{p'}_{(\mathrm{div},0)}(\Omega) } \frac{-\int_{\Omega} \mb{u}\cdot \Div \bs{\tau}_0 \dx }{\norm{\bs{\tau}_0}_{\Hdiv}} - \frac{\int_{\Omega} \hat{\mb{u}} : \bs{\tau_0} \dx}{\norm{\bs{\tau}_0}_{\Hdiv}} \nonumber \\
	&\geq \sup_{\bs{\tau} \in \Hdiv } \frac{-\int_{\Omega} \mb{u}\cdot \Div \bs{\tau} \dx }{\norm{\bs{\tau}}_{\Hdiv}} - \frac{\int_{\Omega} \hat{\mb{u}} : \bs{\tau} \dx}{\norm{\bs{\tau}}_{\Hdiv}}. 
	\end{align}
We define $\mb{z} \in \mb{W}_{0}^{1,p}(\Omega)$ as the unique solution of the weak formulation (see \cite[Lemma 3.5]{gatica2021}, also  \cite{ciarlet1978})
	\begin{align*}
	\int_{\Omega} \abs{\Eps \mb{z}}^{p-2} \Eps \mb{z} : \Eps \mb{v} \dx = -\int_\Omega \abs{\mb{u}}^{p-2}\mb{u} \cdot \mb{v}, \quad \forall \mb{v} \in \mb{W}_{0}^{1,p}(\Omega).
	\end{align*}
	Then owing to Korn's inequality (see \cite[Lemma 3.3]{Jouvet}), $\norm{\mb{z}}_{\mb{W}^{1,p}(\Omega)} \leq \norm{\mb{u}}_{\mb{L}^{p}(\Omega)} $. We set $\bs{\tau} = \tilde{\bs{\tau}} \coloneqq \abs{\Eps \mb{z}}^{p-2} \Eps \mb{z}$, hence $\Div(\tilde{\bs{\tau}}) = \abs{\mb{u}}^{p-2}\mb{u}$ and $\norm{\tilde{\bs{\tau}}}_{\Hdiv} \leq C \norm {\mb{u}}_{\mb{L}^{p}(\Omega)}^{p-1}$. In this way, noting that $\tilde{\bs{\tau}}$ is symmetric, from \eqref{eq:c_part1} we find that
	\begin{align} \label{eq:sup_u}
	\sup_{(\bs{\tau},\phi)\in \Hdiv \times L^{p}(\Omega)} \frac{c((\bs{\tau}, \phi),(\mb{u},\hat{\mb{u}},\lambda))}{\norm{(\bs{\tau},\phi)}_{ \Hdiv \times L^{p}(\Omega)}} &\geq    \frac{ \norm{\mb{u}}_{\mb{L}^p(\Omega)}^p}{\norm{\tilde{\bs{\tau}}}_{\Hdiv}} \geq C \norm{\mb{u}}_{\mb{L}^p(\Omega)}.
	\end{align}
	On the other hand, we consider $\mb{w} \in \mb{W}_{0}^{1,p}(\Omega)$ as the unique solution of the following variational equation
	\begin{align*}
	\int_{\Omega} \abs{\Eps \mb{w}}^{p-2} \Eps \mb{w} : \Eps \mb{v} \dx = -\int_\Omega \abs{\hat{\mb{u}}}^{p-2}\hat{\mb{u}} : \Eps \mb{v}, \quad \forall \mb{v} \in \mb{W}_{0}^{1,p}(\Omega).
	\end{align*}
	Similarly, as for the previous case, we set $\tilde{\bs{\tau}} = \abs{\Eps \mb{w}}^{p-2} \Eps \mb{w}$ , and by employing again the Poincaré and Korn inequalities, we have that $\norm{\tilde{\bs{\tau}}}_{\mathbb{L}^{p'}(\Omega)} \leq \norm {\hat{\mb{u}}}_{\mathbb{L}^p(\Omega)}^{p-1}$. Thus, we define $\hat{\bs{\tau}} = \tilde{\bs{\tau}} + \abs{\hat{\mb{u}}}^{p-2}\hat{\mb{u}}$, which we realize is divergence free, since $\hat{\bs{\tau}}:\hat{\mb{u}} = \abs{\hat{\mb{u}}}^{p-2}(\hat{\mb{u}}:\hat{\mb{u}})$, from \eqref{eq:c_part1} we deduce that,
		\begin{align} \label{eq:sup_hatu}
	\sup_{(\bs{\tau},\phi)\in \Hdiv \times L^{p}(\Omega)} \frac{c((\bs{\tau}, \phi),(\mb{u},\hat{\mb{u}},\lambda))}{\norm{(\bs{\tau},\phi)}_{ \Hdiv \times L^{p}(\Omega)}} &\geq    \frac{ \norm{\hat{\mb{u}}}_{\mb{L}^p(\Omega)}^p}{\norm{\hat{\bs{\tau}}}_{\Hdiv}} \geq C \norm{\hat{\mb{u}}}_{\mathbb{L}^p(\Omega)}.
	\end{align}
Finally, the result is a consequence of \eqref{eq:sup_lambd}, \eqref{eq:sup_u} and \eqref{eq:sup_hatu}.
\end{proof}

We are now in position to provide our main result. In order to conclude the unique solvability of \eqref{eq:var_mod}, we will employ that $\mb{A}_{\gamma}$ is strongly monotone and hemicontinuous, and that $b$ and $c$ satisfy appropriate inf-sup conditions.

\begin{thm} \label{thm:wellp_var_mod}
	Let $\Omega \subset \mathbb{R}^d$, then problem \eqref{eq:var_mod} has a unique solution $(\bs{\theta}, \bs{\sigma},\phi, \mb{u}, \hat{\mb{u}}, \lambda) \in \mathbb{L}^{p}(\Omega)\times \Hdiv\times L^{p'}(\Omega) \times \mb{L}^{p}(\Omega) \times \Lskew \times \mathbb{R}$.
\end{thm}
\begin{proof}	
	Proceeding similarly as in the proof of \cite[Theorem 3.1]{ervin2008}, by Lemma \ref{lem:inf-sup-c}, we know that operator $\bs{C}:  \Hdiv \times L^{p}(\Omega) \to (\mb{L}^{p}(\Omega)\times \Lskew \times \mathbb{R})^*$ defined as $\angles{C(\bs{\sigma},\phi),(\mb{v},\hat{\mb{v}},\eta)} \coloneqq c((\bs{\sigma}, \phi),(\mb{v},\hat{\mb{v}},\eta))$ is an isomorphism from $(\Hdiv \times L^{p'}(\Omega)) / \mb{K_1}$ into $(\mb{L}^{p}(\Omega)\times \Lskew \times \mathbb{R})^*$. Then, there exists $(\bs{\sigma}_0, \phi_0) \in (\Hdiv \times L^p(\Omega))/\mb{K_1}$, such that
	\begin{align*}
	c((\bs{\sigma}_0, \phi_0),(\mb{v}, \hat{\mb{v}},\eta)) &= \int_{\Omega} \mb{f}\cdot \mb{v}, \,\,\mbox{for all $(\mb{v}, \hat{\mb{v}},\eta) \in  \mb{L}^{p}(\Omega)\times \Lskew \times \mathbb{R}$},
	\end{align*}
	and, 
	\begin{align*}
\norm{(\bs{\sigma}_0,\phi_0)}_{(\Hdiv \times L^p(\Omega))/\mb{K_1}} \leq (1/C_c) \norm{f}_{\mb{L}^{p'}(\Omega)}.
	\end{align*}
	Let us define the operator $\bs{B}:  \mathbb{L}^p(\Omega) \to (\Hdiv \times L^{p'}(\Omega))^*$ as $\angles{B(\bs{\theta}),(\bs{\tau},\psi)} \coloneqq b(\bs{\theta},(\bs{\tau},\psi))$, $\bs{\sigma} = \tilde{\bs{\sigma}} + \bs{\sigma}_0$ and $\phi = \tilde{\phi} + \phi_0$. Then solving problem \eqref{eq:saddle-point-problem} is equivalent to find $(\bs{\theta},(\tilde{\bs{\sigma}},\tilde{\phi}))\in \mathbb{L}^{p}(\Omega)  \times \mb{K}_2$ such that
\begin{subequations} \label{eq:redprob1}
\begin{equation}
\begin{array}{rcll}
a_{\gamma}(\bs{\theta},\bs{\xi}) + b^t(\bs{\xi},(\tilde{\bs{\sigma}},\tilde{\phi})) &=& -b^t(\bs{\xi},(\bs{\sigma}_0,\phi_0)), &\text{for all } \bs{\xi} \in \mathbb{L}^{p}(\Omega), \\
b(\bs{\theta},(\bs{\tau},\psi)) &=&\cblue{-\int (\bs{\tau}\cdot\bs{n})\cdot\mb{u}^D\ds}, &\text{for all }(\bs{\tau}, \psi) \in \mb{K}_1,
\end{array}
\end{equation}
\end{subequations}
	where $\mb{K}_2 \subset \mathbb{L}^p(\Omega)$ is defined as
	\begin{align*}
	\mb{K_2} &\coloneqq \left\{ \bs{\theta} \in \mathbb{L}^{p}(\Omega) \,:\, b(\bs{\theta},(\bs{\tau},\psi)) = 0, \,\, \forall (\bs{\tau},\psi)\in\mb{K}_1  \right\}.
	\end{align*}
		
	As in the previous case the inf-sup condition given in Lemma \ref{lem:infsup-b} and \cblue{$\mb{u}^D\in \mb{W}^{p',p}(\Gamma)$}, allows us to find $\bs{\theta_0}$ such that
	\begin{align*}
	b(\bs{\theta}_0,(\bs{\tau},\psi)) &=\cblue{-\int (\bs{\tau}\cdot\bs{n})\cdot\mb{u}^D\ds}, \,\,\mbox{for all $(\bs{\tau}, \psi) \in \mb{K}_1$},
	\end{align*}
	 and reduce problem \cblue{\eqref{eq:redprob1}} to the equivalent form: find $\tilde{\bs{\theta}} \in \mb{K_2}$ such that
	\begin{align} \label{eq:redprob2}
	a_{\gamma}(\tilde{\bs{\theta}},\bs{\xi}) &= -b^t(\bs{\xi},(\bs{\sigma}_0,\phi_0)), \,\,\mbox{for all $\bs{\xi} \in \mb{K_2}$}.
	\end{align}
	Now, the result in Lemma \ref{lem:opAmon} \cblue{allows} us to apply the Minty-Browder method of strictly monotone operators (see \cite[Sec. 2.3.3]{JClibro}) to state the existence of a unique $\tilde{\bs{\theta}}$ solving problem \eqref{eq:redprob2}, and this uniquely \cblue{determines} $\bs{\theta} = \tilde{\bs{\theta}} + \bs{\theta}_0$. Lemma \ref{lem:infsup-b} and equation \eqref{eq:spp_1} implies the existence of a unique $(\tilde{\bs{\sigma}},\tilde{\phi}) \in \mb{K}_1$ satisfying:
	\begin{align*}
	b^t(\bs{\xi},(\tilde{\bs{\sigma}},\tilde{\phi})) &= -b^t(\bs{\xi},(\bs{\sigma}_0,\phi_0)) - a_{\gamma}(\tilde{\bs{\theta}},\bs{\xi}), \,\,\mbox{for all $\bs{\xi} \in \mathbb{L}^{p}(\Omega)$},
	\end{align*}
	which in turn uniquely determines $(\bs{\sigma},\phi)$. Finally, Lemma \ref{lem:inf-sup-c} and equation \eqref{eq:spp_2} imply the existence of a unique $(\mb{u},\hat{\mb{u}},\gamma) \in \mb{L}^{p}(\Omega) \times \Lskew \times \mathbb{R}$ such that 
	\begin{align*}
	c^t((\bs{\tau},\psi),(\mb{u}, \hat{\mb{u}},\lambda)) &= -b(\bs{\theta},(\bs{\tau},\psi)), \,\,\mbox{for all $(\bs{\tau}, \psi) \in \Hdiv \times L^{p'}(\Omega)$}.
	\end{align*}
\end{proof}

An alternative characterization of the solution using a multiplier approach can be obtained if we introduce on \eqref{eq:var_mod} an auxiliary tensor $\mathbf{q}_\gamma$, such that
\[
|\bs{\theta}|_\gamma \mathbf{q}_\gamma = \gamma\tau_s \bs{\theta}.
\]
The strategy is a particularly efficient numerical technique for solving viscoplastic
flow problems in which the non-linearity is related to the unknown velocity gradient \cite{JC2012}. With such a characterization, system \eqref{eq:var_mod} can be equivalently formulated as
\begin{subequations} \label{eq:var_mod_q}
\begin{equation}
\begin{array}{lcl}
-\int_{\Omega} \bs{\theta} \,:\,\bs{\tau} \dx - \int_{\Omega} \psi \tr(\bs{\theta}) \dx - \int_{\Omega} \mb{u} \cdot \Div \bs{\tau}- \int_{\Omega} \hat{\mb{u}} : \bs{\tau} \dx \dx\vspace{0.2cm}
\\ \hspace{2cm}+ \lambda \int_{\Omega} \tr(\bs{\tau}) \dx = \cblue{-\int (\bs{\tau}\cdot\bs{n})\cdot\mb{u}^D\ds}, \,\,\,\,\mbox{for all $(\bs{\tau},\psi) \in \Hdiv \times L^{p'}(\Omega)$}, 
\end{array}
\end{equation}
\begin{equation}
\begin{array}{rcl}
\int_{\Omega} \nu(|\bs{\theta}|)\bs{\theta} : \bs{\xi} \dx + \int_{\Omega} \mathbf{q}_{\gamma} : \bs{\xi} \dx\ - \int_{\Omega} \bs{\sigma} : \bs{\xi}\vspace{0.2cm}\\\hspace{4.cm}- \int_{\Omega} \phi \tr(\bs{\xi}) \dx =0, &\mbox{for all $\bs{\xi} \in \mathbb{L}^{p}(\Omega)$},
\end{array}
\end{equation}
\begin{equation}
\begin{array}{rcl}
-\int_{\Omega} \mb{v} \cdot \Div \bs{\sigma} \dx - \int_{\Omega} \hat{\mb{v}}: \bs{\sigma} \dx + \eta \int_{\Omega} \tr(\bs{\sigma}) \,dx \vspace{0.2cm}\\\hspace{0.cm}= \int_{\Omega} \mb{v} \cdot \mb{f} \dx, &\mbox{for all  $(\mb{v},\hat{\mb{v}},\eta) \in  \mb{L}^{p}(\Omega)\times \Lskew \times \mathbb{R}$},
\end{array}
\end{equation}
\begin{equation}
\begin{array}{rcl}
\int_{\Omega} (\gamma \tau_s \bs{\theta} - |\bs{\theta}|_{\gamma} \mb{q}) : \mb{w} \dx  =0, &\mbox{for all $\mb{w} \in \mathbb{L}^{\infty}(\Omega)$}. \label{eq:vmq_q}
\end{array}
\end{equation}
\end{subequations}
The existence of multiplier $\cblue{\mathbf{q}_\gamma}$ is discussed in \cite{JC2012} and references therein. In the following proposition, we proceed to prove the equivalence between the two formulations, and in consequence the existence of a unique solution for problem \eqref{eq:var_mod_q}.

\begin{prop} \label{prop:existence_q}
	Given $\gamma>0$, problem \eqref{eq:var_mod_q}  has a unique solution $(\bs{\theta}, \bs{\sigma},\phi, \mb{u}, \hat{\mb{u}},\mb{q}_{\gamma}, \lambda) \in \mathbb{L}^{p}(\Omega)\times \Hdiv\times L^{p'}(\Omega) \times \mb{L}^{p}(\Omega)\times \Lskew \times \mathbb{L}^{p'}(\Omega)\times \mathbb{R}$.
\end{prop}
\begin{proof}
	The proof is based on the equivalence between \eqref{eq:var_mod} and \eqref{eq:var_mod_q}. Let us suppose that $\{\bs{\theta}, \bs{\sigma},\phi, \mb{u}, \hat{\mb{u}},\lambda\}\in \mathbb{L}^{p}(\Omega)\times \Hdiv\times L^{p'}(\Omega) \times \mb{L}^{p}(\Omega)\times \Lskew \times \mathbb{R}$ is a solution for \eqref{eq:var_mod} and define
	\[
	\mathbf{q}:=\gamma \tau_s \frac{\bs{\theta}}{|\bs{\theta}|_\gamma}.
	\]
	First, note that, since $\bs{\theta} \in \mathbb{L}^{p}(\Omega)$, $\mathbf{q}$ is a matrix of measurable functions. Further, since $|\frac{\bs{\theta}}{|\bs{\theta}|_\gamma}|\leq \gamma$ a.e. in $\Omega$, we conclude that $\mathbf{q}\in \mathbb{L}^{\infty}(\Omega)$ and that $\mathbf{q}|\bs{\theta}|_\gamma^{-1}=\gamma \tau_s\bs{\theta}$ and $\mathbf{q}=\gamma\tau_s \frac{\bs{\theta}}{|\bs{\theta}|_\gamma}$ hold in $\mathbb{L}^{p'}(\Omega)$.
	Thus, with this definition for $\mathbf{q}$, we have that
	\[
	a_\gamma(\bs{\theta},\bs{\psi})= a(\bs{\theta},\bs{\psi}) + (\mathbf{q}\,,\,\bs{\psi})_{p'} = a(\bs{\theta},\bs{\psi}) + \gamma \tau_s \int_\Omega \frac{\bs{\theta}\,:\,\bs{\psi}}{|\bs{\theta}|_\gamma}\, dx,
	\]
	which implies that $\{\bs{\theta}, \bs{\sigma},\phi, \mb{u}, \hat{\mb{u}}, \lambda\}$ satisfies \eqref{eq:var_mod} for any $\{\bs{\tau}, \bs{\xi},\psi, \mb{v},\hat{\mb{v}}, \eta\}\in \mathbb{L}^{p}(\Omega)\times \Hdiv\times L^{p'}(\Omega) \times \mb{L}^{p}(\Omega) \times \Lskew \times \mathbb{R}$. Consequently, the existence of a solution for \eqref{eq:var_mod_q} follows from Theorem \ref{thm:wellp_var_mod}. 
	
	Reciprocally, let us assume that $\{\bs{\theta}, \bs{\sigma},\phi, \mb{u},\hat{\mb{u}}, \mb{q}, \lambda\}$ solve \eqref{eq:var_mod_q}. From \eqref{eq:vmq_q}, we have that
	\[
	\int_\Omega \left(\gamma\tau_s \bs{\theta} - |\bs{\theta}|_\gamma \mathbf{q}\right)\,:\,\mathbf{w} \, dx =0,\,\,\forall \mathbf{w}\in \mathbb{L}^{\infty}(\Omega).
	\]
	This expression implies that $\gamma\tau_s \bs{\theta} - |\bs{\theta}|_\gamma \mathbf{q}=0$ in $\mathbb{L}^{1}(\Omega)$, and for any solution of \eqref{eq:var_mod_q}, we have that
	\[
	\mathbf{q}= \gamma\tau_s \frac{\bs{\theta}}{|\bs{\theta}|_\gamma},\,\,\mbox{ a.e. in $\Omega$}.
	\]
	Using this expression in \eqref{eq:var_mod_q}, with $\mathbf{w}=0$, yields that $\{\bs{\theta}, \bs{\sigma},\phi, \mb{u}, \hat{\mb{u}},\lambda\}$ solve \eqref{eq:var_mod}, so the uniqueness follows from Theorem \ref{thm:wellp_var_mod} and the given definition of $\mathbf{q}$.
\end{proof}

\section[Discretization and Linearization]{Discretization and Linearization} \label{sec:numerical}
In this section, we discuss the FEM discretization of \eqref{eq:var_mod_q} and propose a linearization based on the semismooth Newton methods.

\subsection{FEM Discretization}
Let us assume that $\Omega\subset\Real^d$ is a convex and \cblue{polyhedral} domain, and $\mathcal{T}_h$ is a regular mesh of $\overline{\Omega}$, parametrized by $h:=\max_{T\in \mathcal{T}_h}\textrm{diag}\,T$. We look for finite dimensional spaces $\mb{T}_h \subset \mathbb{L}^p(\Omega)$, $\mb{T}_{h,\mathrm{div}} \subset \Hdiv$, $\mb{T}_{h,\mathrm{sk}} \subset \Lskew$, $Q_h\subset L_0^{p'}(\Omega)$, $\mathbf{V}_h\subset \mathbf{W}_0^{1,p}(\Omega)$,  and $\mathbf{W}_h\subset \mathbb{L}^{p'}(\Omega)$, which allow us to propose the following FEM approximation of \eqref{eq:var_mod_q}: find $\bs{\theta}_h \in \mb{T_h}$, $\bs{\sigma}_h \in \bs{T}_{h,\mathrm{div}}$, $\mathbf{u}_h\in \mathbf{V}_h$, $\hat{\mb{u}} \in  \bs{T}_{h,\mathrm{sk}}$, $\phi_h\in Q_h$, $\lambda_h \in \mathbb{R}$ and $\mathbf{q}_h\in \mathbf{W}_h$ such that
\begin{subequations} \label{eq:varmdis}
	\begin{align}
	-\int_{\Omega} \bs{\theta}_h \,:\,\bs{\tau}_h \dx - \int_{\Omega} \psi_h \tr(\bs{\theta}_h) \dx - \int_{\Omega} \mb{u}_h \cdot \Div \bs{\tau}_h \dx
	\nonumber \\  - \int_{\Omega} \hat{\mb{u}}_h : \bs{\tau}_h \dx + \lambda_h \int_{\Omega} \tr(\bs{\tau}_h) \dx &= \cblue{-\int (\bs{\tau_h}\cdot\bs{n})\cdot\mb{u}^D\ds}\\
	\int_{\Omega} \nu(|\bs{\theta}_h|)\bs{\theta}_h : \bs{\xi}_h \dx + \int_{\Omega} \mathbf{q}_{h} : \bs{\xi}_h \dx - \int_{\Omega} \bs{\sigma}_h : \bs{\xi}_h \nonumber \\
	- \int_{\Omega} \phi_h \tr(\bs{\xi}_h) \dx &=0 \\
	-\int_{\Omega} \mb{v}_h \cdot \Div \bs{\sigma}_h \dx - \int_{\Omega} \hat{\mb{v}}_h: \bs{\sigma}_h \dx \\+ \eta_h \int_{\Omega} \tr(\bs{\sigma}_h) \dx &= \int_{\Omega} \mb{v}_h \cdot \mb{f} \dx \\
	 \int_{\Omega} (\gamma \tau_s \bs{\theta}_h - |\bs{\theta}_h|_{\gamma} \mb{q}_h) : \mb{w}_h \dx & =0 ,
	\end{align}
\end{subequations}

for all $(\bs{\tau}_h,\psi_h) \in (\mb{T}_{h,\mathrm{div}}\times Q_h)$, $\bs{\xi}_h \in \mb{T}_h$,  $(\mb{v}_h,\hat{\mb{v}}_h,\eta_h) \in  \mb{V}_h\times \mb{T}_{h,\mathrm{sk}}\times \mathbb{R}$ and $\mb{w}_h \in \mb{W}_h$.

\subsection{Well-posedness of the Galerkin scheme}\label{sec:galerkinwp}

It is clear that no restrictions have to be added to $\mb{T}_h$, $\mb{T}_{h,\mathrm{div}}$,$\mb{T}_{h,\mathrm{sk}}$, $Q_h$, $\mb{V}_h$ and $\mb{W}_h$ other than being finite dimensional subspaces of the described spaces. However, for ellipticity purposes of the operator induced by $b$ in $\mb{K}_1$ and the operator induced by $c$, the following inf-sup conditions must be met:

\begin{enumerate}[label={(\textbf{H.\arabic*})}]
	\item \label{hd1}There exists a constant $\hat{C}_b$, independent of $h$ such that
	\begin{align*}
\inf_{(\bs{\tau}_h,\psi_h) \in \mb{K}_{1,h}}  \, \sup_{\bs{\theta}_h \in \mb{T}_h} \frac{b(\bs{\theta}_h,(\bs{\tau}_h,\psi_h))}{\norm{\bs{\theta}_h}_{\mathbb{L}^{p}(\Omega)} \norm{(\bs{\tau}_h,\psi_h)}_{ \Hdiv\times L^{p'}(\Omega)}} &\geq \hat{C}_b.
\end{align*}	
\item \label{hd2} Let $\mb{\mathcal{V}}_h \coloneqq \mb{V}_h\times \mb{T}_{h,\mathrm{sk}}\times \mathbb{R}$ and $\mb{\mathcal{Q}}_h \coloneqq \mb{T}_{h,\mathrm{div}} \times Q_h$. There exists a constant $\hat{C}_c$, independent of $h$ such that
\begin{align*}
\inf_{(\mb{v}_h,\hat{\mb{v}}_h,\eta_h) \in \mb{\mathcal{V}}_h }  \, \sup_{(\bs{\sigma}_h,\phi_h)\in \mb{\mathcal{Q}}_h} \frac{c((\bs{\sigma}_h, \phi_h),(\mb{v}_h,\hat{\mb{v}}_h,\eta_h))}{\norm{(\bs{\sigma}_h,\phi_h)}_{ \Hdiv \times L^{p}(\Omega)} \norm{(\mb{v}_h,\hat{\mb{v}}_h,\eta_h)}_{\mb{L}^{p}(\Omega)\times \Lskew \times \mathbb{R}}} &\geq \hat{C}_c,
\end{align*}
\end{enumerate}
where the discrete kernels of $B$ and $C$ are given by:
\begin{align*}
\mb{K}_{1,h} &\coloneqq \left\{ (\bs{\sigma}_h, \phi_h)\in \mb{T}_{h,\mathrm{div}} \times Q_h \,:\, c((\bs{\sigma}_h, \phi_h),(\mb{v}_h,\hat{\mb{v}}_h, \eta_h)) = 0 \right\}, \\
\mb{K}_{2,h} &\coloneqq \left\{ \bs{\theta}_h \in \mb{T}_h \,:\, b(\bs{\theta}_h,(\bs{\tau}_h,\psi_h)) = 0 \quad \forall (\bs{\tau}_h,\psi_h)\in\mb{K}_{1,h}  \right\}.
\end{align*}

\begin{prop}
Assume \ref{hd1} and \ref{hd2}. Then problem \eqref{eq:varmdis} has a unique solution \\ $(\bs{\theta}_h, \bs{\sigma}_h, \mathbf{u}_h,\hat{\mb{u}}_h,\phi_h,\lambda_h,\mathbf{q}_h)$ $\in \mb{T}_h\times\mb{T}_{h,\mathrm{div}}\times\mathbf{V}_h\times\mb{T}_{h,\mathrm{sk}}\times Q_h\times \mathbb{R}\times\mathbf{W}_h$.
\end{prop}
\begin{proof}
	Existence and uniqueness of  $(\bs{\theta}_h, \bs{\sigma}_h, \mathbf{u}_h,\hat{\mb{u}}_h,\phi_h,\lambda_h,\mathbf{q}_h) \in \mb{T}_h\times\mb{T}_{h,\mathrm{div}}\times\mathbf{V}_h\times \mb{T}_{h,\mathrm{sk}} \times Q_h\times \mathbb{R}\times\mathbf{W}_h$ solving \eqref{eq:varmdis} follows the same arguments from Theorem \ref{thm:wellp_var_mod}, and Proposition \ref{prop:existence_q}, replacing the continuous inf-sup conditions by their discrete versions.
\end{proof}
\subsection{Specific finite element subspaces}

Given a set $D \subset \mathbb{R}^d$ and an integer $k \geq 0$, we define $P^k(D)$ as the space of polynomial functions on $D$ of degree up to $k$, with vector and tensorial versions denoted by $\mb{P}^k(D) \coloneqq [P^k(D)]^d$ and $\mathbb{P}^k(D) \coloneqq [P^k(D)]^{d\times d}$. Hence, problem \eqref{eq:varmdis} can be approximated using the following finite element subspaces:
\begin{align*}
\mb{T}_h &\coloneqq \left\{ \bs{\sigma} \in \Hdiv \,:\, \bs{\sigma}|_{T} \in \mathbb{P}^{k+1}(T), \quad \forall T \in \mathcal{T}_h \right\}, \\
\mb{T}_{h,\mathrm{div}} &\coloneqq \left\{ \bs{\phi} \in \Hdiv \,:\, \bs{\phi}|_{T} \in \mathbb{P}^{k+1}(T), \quad \forall T \in \mathcal{T}_h \right\}, \\
\mb{T}_{h,\mathrm{sk}} &\coloneqq \left\{ \bs{\psi} \in \Lskew \,:\, \bs{\psi} \in \bs{\psi}|_{T} \in \mathbb{P}^{k}(T), \quad \forall T \in \mathcal{T}_h \right\}, \\
\mb{V}_h &\coloneqq \left\{ \mb{v} \in \mb{L}^p(\Omega) \,:\, \mb{v}|_{T} \in \mb{P}^{k}(T), \quad \forall T \in \mathcal{T}_h \right\}, \\
Q_h &\coloneqq \left\{ \phi \in L^{p'}(\Omega) \,:\, \phi|_{T} \in P^k(T), \quad \forall T \in \mathcal{T}_h \right\}, \\
\mb{W}_h &\coloneqq \left\{ \mb{w} \in \mathbb{L}^{p'}(\Omega) \,:\, \mb{w}|_{T} \in \mathbb{P}^{k+1}(T), \quad \forall T \in \mathcal{T}_h \right\}.
\end{align*}
Here the discrete product space, $\mb{T}_{h,\mathrm{div}}\times \mb{T}_{h,\mathrm{sk}} \times \mb{V}_h$ constitutes the element spaces of Arnold, Falk and Winther (AFW) introduced and proved to be stable in \cite{arnold2007} for the mixed finite element approximation of Dirichlet linear elasticity. Furthermore, we remark that the spaces $\mb{T}_h$, $\mb{T}_{h,\mathrm{div}}$,$\mb{T}_{h,\mathrm{sk}}$, $\mathbf{V}_h$ and $Q_h$ satisfy the inf-sup conditions \ref{hd1} and \ref{hd2}. We remit to the Banach spaces-based analysis in \cite[Section 4.4]{gatica2021} (see also \cite[Theorem 5.1]{Barrientos2002}, \cite[Section 4]{Farhloul2017}, \cite[Lemma 4.2 and Lemma 4.3]{ervin2008}) for further details.

\subsection{SSN Linearization}
Once we have formulated problem \eqref{eq:varmdis}, we propose a SSN linearization for this system. This is enabled thanks to the fact that both the Frobenius norm and the max functions are semismooth in finite dimension spaces. 

The SSN linearization about $(\bs{\theta}_h, \bs{\sigma}_h, \mathbf{u}_h,\hat{\mb{u}}_h,\phi_h,\lambda_h,\mathbf{q}_h)$ gives the following problem: find $\delta_{\bs{\theta}}\in \mb{T_h}$, $\delta_{\bs{\sigma}} \in \bs{T}_{h,\mathrm{div}}$, $\delta_{\mathbf{u}}\in \mathbf{V}_h$, $\delta_{\hat{\mb{u}}} \in  \bs{T}_{h,\mathrm{sk}}$, $\delta_{\phi}\in Q_h$, $\delta_{\lambda} \in \mathbb{R}$ and $\delta_{\mathbf{q}}\in \mathbf{W}_h$, such that
\begin{subequations}\label{SSN}
\begin{equation}\label{SSN1}
\begin{array}{lll}
-\int_\Omega \delta_{\bs{\theta}}\,:\,\bs{\tau}_h\,dx -\int_\Omega \psi_h\, \tr( \delta_{\bs{\theta}})\,dx -\int_\Omega  \delta_{\mathbf{u}}\cdot \Div \bs{\tau}_h\,dx -\int_\Omega  \delta_{\hat{\mathbf{u}}}\,:\,\bs{\tau}_h\,dx\vspace{0.2cm}\\\hspace{2.5cm} +\delta_\lambda\int_\Omega \tr(\bs{\tau}_h)\,dx=\cblue{-\int (\bs{\tau}_h\cdot\bs{n})\cdot\mb{u}^D\ds}+\int_\Omega \bs{\theta}_h\,:\,\bs{\tau}_h\,dx + \int_\Omega \psi_h \tr(\bs{\theta}_h)\,dx  \vspace{0.2cm}\\\hspace{5cm}+ \int_\Omega \mathbf{u}_h\cdot \Div \bs{\tau}_h\,dx+ \int_\Omega \hat{\mathbf{u}}_h\,:\,\bs{\tau}_h\,dx - \lambda_h\int_\Omega \tr(\bs{\tau}_h)\,dx,
\end{array}
\end{equation}
\begin{equation}
\begin{array}{lll}
\int_\Omega \frac{\nu'(|\bs{\theta}_h|)}{|\bs{\theta}_h|}(\bs{\theta}_h\,:\,\delta_{\bs{\theta}})\,(\bs{\theta}_h\,:\,\bs{\xi}_h)\,dx +\int_\Omega \nu(|\bs{\theta}_h|)(\delta_{\bs{\theta}}\,:\,\bs{\xi}_h)\,dx +\int_\Omega  \delta_{\mathbf{q}}\,:\, \bs{\xi}_h\,dx \vspace{0.2cm}\\\hspace{2cm} -\int_\Omega  \delta_{\bs{\sigma}}\,:\,\bs{\xi}_h\,dx -\int_\Omega \delta_{\phi}\tr(\bs{\xi}_h)\,dx=- \int_\Omega \nu(|\bs{\theta}_h|)(\bs{\theta}_h\,:\,\bs{\xi}_h)\,dx - \int_\Omega \mathbf{q}_h\,:\, \bs{\xi}_h\,dx  \vspace{0.2cm}\\\hspace{6cm}+ \int_\Omega \bs{\sigma}_h\,:\,\bs{\xi}_h\,dx+ \int_\Omega \phi_h\tr(\bs{\xi}_h)\,dx,
\end{array}
\end{equation}
\begin{equation}
\begin{array}{lll}
-\int_\Omega \mathbf{v}_h\cdot \Div\delta_{\bs{\sigma}}\,dx -\int_\Omega \hat{\mathbf{v}}_h\,:\,\delta_{\bs{\sigma}}\,dx +\eta_h\int_\Omega  \tr(\delta_{\bs{\sigma}})\,dx \vspace{0.2cm}\\\hspace{1cm} = \int_\Omega \mathbf{v}_h\cdot \Div\bs{\sigma}_h\,dx +  \int_\Omega \hat{\mathbf{v}}_h\,:\, \bs{\sigma}_h\,dx  -\eta_h\int_\Omega \tr(\bs{\sigma}_h)\,dx+ \int_\Omega \mathbf{v}_h\cdot\mathbf{f}\,dx,
\end{array}
\end{equation}
\begin{equation}\label{SSN4}
\begin{array}{lll}
\gamma\tau_s\int_\Omega \delta_{\bs{\theta}}\,:\mathbf{w}_h\,dx -\gamma\int_\Omega\frac{\chi_{\mathcal{A}_\gamma}}{|\bs{\theta}_h|} \left(\bs{\theta}_h\,:\,\delta_{\bs{\theta}}\right)\,(\mathbf{q}_h\,:\,\mathbf{w}_h)\,dx  -\int_\Omega |\bs{\theta}_h|_\gamma\,\delta_{\mathbf{q}}\,:\,\mathbf{w}_h\,dx\vspace{0.2cm}\\\hspace{5cm} = -\gamma\tau_s \int_\Omega \bs{\theta}_h\,:\,\mathbf{w}_h\,dx + \int_\Omega |\bs{\theta}_h|_\gamma \mathbf{q}_h\,:\,\mathbf{w}_h\,dx.
\end{array}
\end{equation}
\end{subequations}
for all $(\bs{\tau}_h,\psi_h) \in (\mb{T}_{h,\mathrm{div}}\times Q_h)$, $\bs{\xi}_h \in \mb{T}_h$,  $(\mb{v}_h,\hat{\mb{v}}_h,\eta_h) \in  \mb{V}_h\times \mb{T}_{h,\mathrm{sk}}\times \mathbb{R}$ and $\mb{w}_h \in \mb{W}_h$. Here, we have that
\[
\chi_{\mathcal{A}_\gamma}:=\left\{
\begin{array}{lll}
1& \mbox{if $|\bs{\theta}_h|\geq \frac{\tau_s}{\gamma}$}\vspace{0.2cm}\\0&\mbox{otherwise}.
\end{array}
\right. 
\]

\subsection{Remarks on the Convergence of the SSN Method}
We follow the ideas in \cite{Sun-Han} in order to analyse the convergence of the SSN linearization \eqref{SSN}. Before going on and for the article's completeness, we recall the definition of slantly differentiable function (see \cite{JC2010, Nashed}).
\begin{defi}
\cblue{Let $X$ and $Y$ be two Banach spaces, and let $D\subset X$ be an open domain. A function $F:D\subset X\rightarrow Y$ is said to be slantly differentiable at $x\in D$ if there exists a mapping $G_F:D\rightarrow\mathcal{L}(X,Y)$ such that the family $\{G_F(x+h)\}$ of bounded linear operators is uniformly bounded in the operator norm for $h$ sufficiently small and
\[
\underset{h\rightarrow 0}{\lim}\frac{F(x+h)-F(x)- G_F(x+h) h}{\|h\|}=0.
\]
}
\end{defi}
Let us start the discussion on the convergence of the SSN iteration, by noticing that linearization \eqref{SSN} was proposed for the discretized system \eqref{eq:varmdis}. Therefore, since we are in a finite dimensional setting, the Huber term $|\bs{\theta}|_\gamma$ is semismooth and slantly differentiable (see \cite{JC2010, JC2012} and the references therein).  Hence,  the left hand side in the system \eqref{SSN} is well defined. 

We now discuss the existence of solutions for system \eqref{SSN}. As it was proved in Section \ref{sec:galerkinwp}, the \cblue{discretized} system \eqref{eq:varmdis} is well posed, which guarantees that the FEM approach induces positive definite matrices associated to the forms in the variational system. Thus, the matrix formulation involves matrices which are always non singular and the system \eqref{SSN} has a unique solution. This fact is clear in the case of coercive bilinear forms involving scalar products such as the Frobenius product.  However, for the following nonlinear forms, the positive definiteness needs to be discussed.
\begin{equation}
\begin{array}{rcl}
a_{\bs{\theta}}(\bs{\zeta}_h,\bs{\xi}_h)&:=&\int_\Omega \nu(|\bs{\theta}_h|)(\bs{\zeta}_h\,:\,\bs{\xi}_h)\,dx + \int_\Omega \frac{\nu'(|\bs{\theta}_h|)}{|\bs{\theta}_h|}(\bs{\theta}_h\,:\,\bs{\zeta}_h)\,(\bs{\theta}_h\,:\,\bs{\xi}_h)\,dx \vspace{0.2cm}\\ c_{\mathbf{q},\bs{\theta}}(\bs{\zeta}_h,\mathbf{w}_h)&:=&\gamma\tau_s\int_\Omega \bs{\zeta}_h\,:\mathbf{w}_h\,dx-\gamma\int_\Omega\frac{\chi_{\mathcal{A}_\gamma}}{|\bs{\theta}_h|} \left(\bs{\theta}_h\,:\,\bs{\zeta}_h\right)\,(\mathbf{q}_h\,:\,\mathbf{w}_h)\,dx.
\end{array}
\end{equation}
Let us prove that these forms are indeed positive definite, which immediately \cblue{yields} that the system has a unique solution.  We start by focusing on $a_{\bs{\theta}}(\bs{\zeta}_h,\bs{\xi}_h)$.  First, note that \cblue{$\nu(t)=\frac{\cblue{\widetilde{\nu}\,'}(t)}{t}$}, which, thanks to \eqref{Delta2}, yields that $\nu(|\bs{\theta}_h|)>0$. Next, we introduce the following operator
\[
\langle \mathbf{A}(\bs{\theta}_h)\,,\,\bs{\xi}_h\rangle:=\int_\Omega \nu(|\bs{\theta}_h|)(\bs{\theta}_h\,:\,\bs{\xi}_h)\,dx.
\]
By following ideas in Lemma \ref{lem:opAmon}, we can state that there exists $C>0$, such that
\[
\langle \mathbf{A}(\bs{\theta}_h)- \mathbf{A}(\bs{\xi}_h)\,,\,\bs{\theta}-\bs{\xi}_h\rangle\geq C \int_\Omega\cblue{\widetilde{\nu}\,''}\left(|\bs{\theta}_h|+|\bs{\xi}_h|\right)|\bs{\theta}_h-\bs{\xi}_h|^2\,dx,\,\,\forall\,\bs{\theta}_h,\bs{\xi}_h\in \mathbf{T}_h.
\]
Next,  by considering $t\in\Real$ and $\bs{\upsilon}_h\in \mathbf{T}_h$, we have that
\[
\left\langle \frac{\mathbf{A}(\bs{\theta}_h+t\bs{\upsilon}_h)-\mathbf{A}(\bs{\theta}_h)}{t}\,,\,\bs{\upsilon}_h\right\rangle \geq \int_\Omega \cblue{\widetilde{\nu}\,''}\left(|\bs{\theta}_h+ t\bs{\upsilon}_h|+|\bs{\xi}_h|\right)|\bs{\upsilon}_h|^2\,dx.
\]
Finally, by noticing that $\cblue{\widetilde{\nu}\,''}(t)$ is a continuous function such that $\cblue{\widetilde{\nu}\,''(t)\approx\frac{\widetilde{\nu}\,'(t)}{t}}>0$, and taking limits $t\rightarrow 0$, we obtain that
\[
\begin{array}{rcl}
\left\langle D\mathbf{A}(\bs{\theta}_h)\bs{\upsilon}_h\,,\,\bs{\upsilon}_h\right\rangle &=& \int_\Omega \nu(|\bs{\theta}_h|)(\bs{\upsilon}_h\,:\,\bs{\upsilon}_h)\,dx+ \int_\Omega \frac{\nu'(|\bs{\theta}_h|)}{|\bs{\theta}_h|}(\bs{\theta}_h\,:\,\bs{\upsilon}_h)\,(\bs{\theta}_h\,:\,\bs{\upsilon}_h)\,dx\vspace{0.2cm}\\&\geq& \int_\Omega\cblue{\widetilde{\nu}\,''}\left(|\bs{\theta}_h|+|\bs{\xi}_h|\right)|\bs{\upsilon}_h|^2\,dx>0.
\end{array}
\]
Thus, we conclude that the form $a_{\bs{\theta}}(\bs{\zeta}_h,\bs{\xi})$ is positive definite.

In the case of the form $c_{\mathbf{q},\bs{\theta}}(\bs{\zeta}_h,\mathbf{w}_h)$, it is known, by following the arguments in \cite[Sec. 6.1]{JC2010}, that a sufficient condition for this form to be positive definite is that $|\mathbf{q}|\leq \tau_s$, a.e. in $\Omega$. Unfortunately, this condition does not necessarily hold. Consequently, we propose a projection procedure on the multiplier $\mathbf{q}_h$ to guarantee that $|\mathbf{q}_h|\leq \tau_s$ a.e. in $\Omega$. Thus, we replace $\mathbf{q}_h$ by
\[
\widehat{\mathbf{q}}_h:=\frac{\tau_s}{\max\{\tau_s,|\mathbf{q}_h|\}}\mathbf{q}_h
\]
when assembling the left hand side of \eqref{SSN}, resulting in the following \cblue{slight} modification of the equation \eqref{SSN4}
\begin{equation}\label{SSN4t}\tag{\ref{SSN4}'}
\begin{array}{lll}
\gamma\tau_s\int_\Omega \delta_{\bs{\theta}}\,:\mathbf{w}_h\,dx  -\gamma\int_\Omega\frac{\chi_{\mathcal{A}_\gamma}}{|\bs{\theta}_h|} \left(\bs{\theta}_h\,:\,\delta_{\bs{\theta}}\right)\,(\widehat{\mathbf{q}}_h\,:\,\mathbf{w}_h)\,dx -\int_\Omega |\bs{\theta}_h|_\gamma\,\delta_{\mathbf{q}}\,:\,\mathbf{w}_h\,dx \vspace{0.2cm}\\\hspace{5cm}= -\gamma\tau_s \int_\Omega \bs{\theta}_h\,:\,\mathbf{w}_h\,dx + \int_\Omega |\bs{\theta}_h|_\gamma \mathbf{q}_h\,:\,\mathbf{w}_h\,dx.
\end{array}
\end{equation}

Now, let us turn our attention to the convergence analysis.  First, we consider the system \eqref{eq:varmdis} as a nonlinear variational equation such as
\[
\bs{\Psi}^h (\bs{\omega})=0,
\]
where $\bs{\omega}$ stands for all the unknowns in the system.  Therefore, the system \eqref{SSN1}-\eqref{SSN4t} represents a semismooth Newton iteration of the form
\begin{equation}\label{SSNit}
\bs{\Xi}^h(\bs{\hat{\omega}}) \delta_{\bs{\omega}} = -\bs{\Psi}^h(\bs{\omega}),
\end{equation}
with $\bs{\Xi}^h(\bs{\hat{\omega}})\in \partial\bs{\Psi}^h(\bs{\omega})$ and $\bs{\hat{\omega}}$ representing the modified unknowns, taking into account the projection step on $\mathbf{q}_h$.  Since we are working with the discretized system \eqref{eq:varmdis}, we can state that all the functions involved in this system are slantly differentiable or semismooth. In particular, the $\max$ and norm functions in \eqref{SSN4t} are known to be semismooth \cite{JC2010}. Therefore, the function $\bs{\Psi}^h(\bs{\omega})$, which represents the variational FEM discretization, is semismooth. Further, the fact that all the main forms in the left hand side of \eqref{SSN} are positive definite, guarantees that all $\bs{\Xi}^h(\bs{\hat{\omega}})\in \partial\bs{\Psi}^h(\bs{\omega})$  are nonsingular.  Finally, by using the same argumentation as in \cite[Sec. 6.1]{JC2010}, we can prove that the projected dual variable $\hat{\mathbf{q}}^h$ converges to the actual $\mathbf{q}^h$ as long as the SSN iteration is developed, and, consequently, $\bs{\Xi}^h(\bs{\hat{\omega}})\rightarrow \bs{\Xi}^h(\bs{\omega})$. This fact implies that all the \cblue{hypotheses} in \cite[Th. 4.2.]{Sun-Han} are fulfilled, which \cblue{yields} the \cblue{local} superlinear convergence of the proposed SSN iteration \eqref{SSN}. \cblue{Although Theorem \cite[Th. 4.2.]{Sun-Han} guarantees only a local superlinear convergence rate, the algorithm does achieves convergence for different initialization steps, needing a few more iterations, depending on the case.  This behaviour can be explained with the projection step, since it has a globalization effect in the SSN loop (see \cite{JC2010}). We will discuss this effect, in the next section.}

\section{Numerical Results} \label{sec:results}

In this section, we present numerical experiments which illustrate the main properties of the algorithm.  All numerical routines have been carried out using the open-source finite element library FEniCS \cite{alnaes} and polynomial degree \cblue{$k=0$ and} $k=1$. Regarding the implementation of the semismooth Newton iterative method, the solution of all linear systems appearing at each SSN iteration is conducted with the multifrontal massively parallel sparse direct solver MUMPS, and the iterations are terminated once the relative error \cblue{computed as the $\ell^2$ norm of the entire residual coefficient vectors ($\delta_{\bs{\theta}},\delta_{\bs{\sigma}},\delta_{\mb{u}},\delta_{\hat{\mb{u}}},\delta_{\phi},\delta_{\lambda},\delta_{\mb{q}}$) obtained from problem \eqref{SSN}}, is sufficiently small, i.e.,
\begin{align*}
\frac{\norm{\mathrm{res}^{n}}_{\ell^2}}{\norm{\mathrm{res}^{0}}_{\ell^2}} \leq \mathrm{tol},
\end{align*}
where $\norm{\cdot}_{\ell^2}$ is the standard $\ell^2$ norm and $\mathrm{tol}$ is a fixed tolerance chosen as $\mathrm{tol} = \num{1e-5}$ \cblue{or $\mathrm{tol} = \num{1e-10}$ when testing convergence}.

\begin{figure}[!t]
	\begin{center}
		\includegraphics[width=0.5\textwidth]{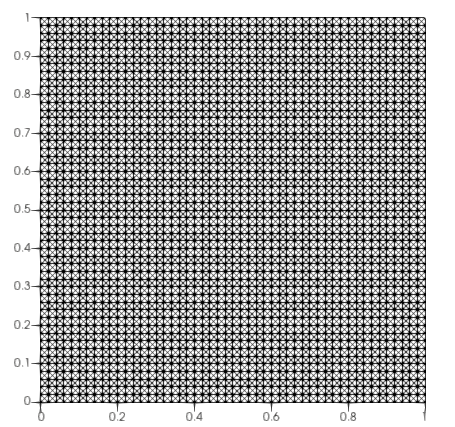}
	\end{center}
	\caption{Crossed pattern triangulation used for two-dimensional numerical tests.}  \label{img:wireframe2d}
\end{figure}

In the examples below, the SSN iteration is initialized with the solution of a similar dual mixed formulation for the discrete Stokes problem. 

\subsection{Reservoir flow}
\begin{figure}[!t]
	\begin{center}
		\includegraphics[width=0.4\textwidth]{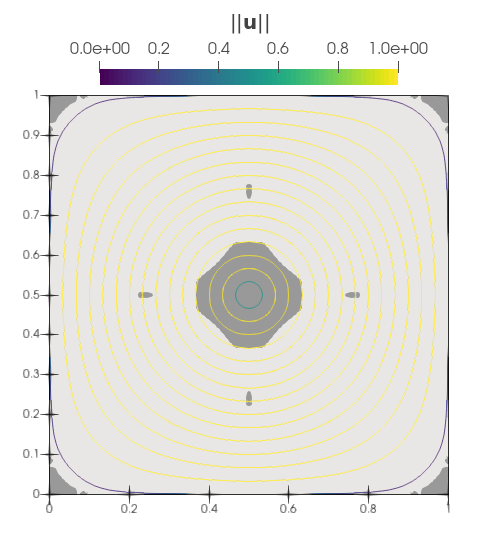}
		\includegraphics[width=0.4\textwidth]{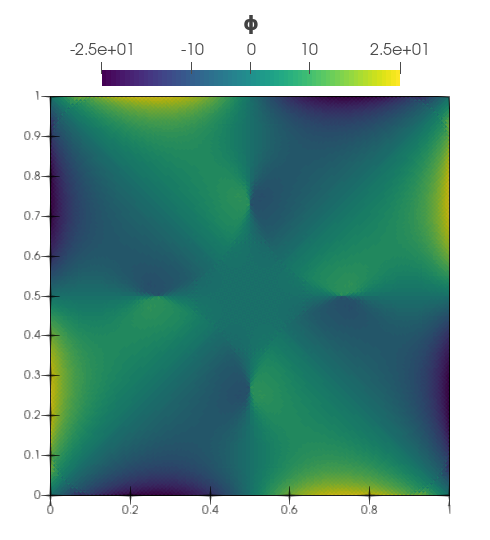}\\
		\includegraphics[width=0.4\textwidth]{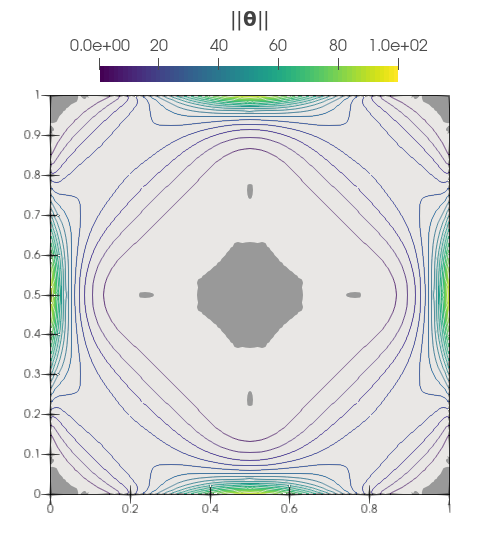}
	\end{center}
	\caption{Reservoir flow: \cblue{Top, left: velocity stream lines and Inactive set (dark gray)}; Top, right: pressure magnitude; Bottom: \cblue{strain tensor magnitude and Inactive set.} Parameters: Herschel-Bulkley model $p=1.75$, $h=1/100$, $\gamma = \num{1e3}$, $\mu=1.0$, \cblue{$\tau_s=10.0$}.}  \label{img:resTest1}
\end{figure}
In our first numerical example, we calculate the reservoir flow in the unit square $\Omega_h \coloneqq (0,1) \times (0,1)$, considering homogeneous Dirichlet boundary conditions and
\begin{align*}
	\mb{f}(x_1,x_2) \coloneqq 300(x_2 - 0.5, 0.5 - x_1).
\end{align*}
We study the Herschel-Bulkley fluid given by $p=1.75$, $\tau_s = 10$ and $\mu = 1$, and consider $k=0$ and a mesh step size \cblue{$h = 1/100$ (40 000 cells) in a \textit{crossed} pattern, that helps with the symmetry of the numerical results (see figure \ref{img:wireframe2d}).}

In Fig. \ref{img:resTest1} the velocity vector fields, the flow streamlines, pressure, strain tensor, and the computed active and inactive sets, are depicted. As expected, we observe a clockwise rotation of the velocity field and a symmetric pattern for all relevant variables.
	\begin{figure}[!t]
		\begin{center}
			\includegraphics[width=0.4\textwidth]{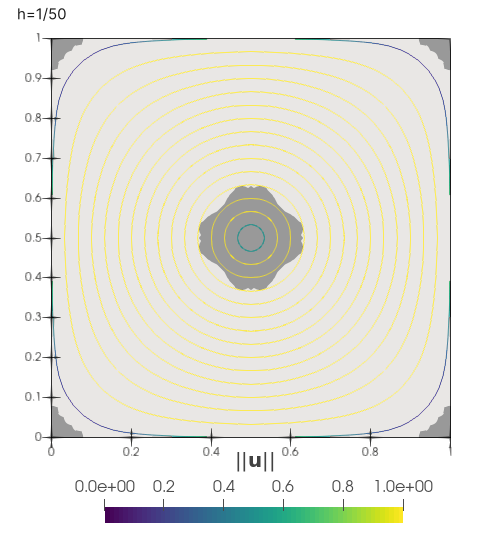}\includegraphics[width=0.4\textwidth]{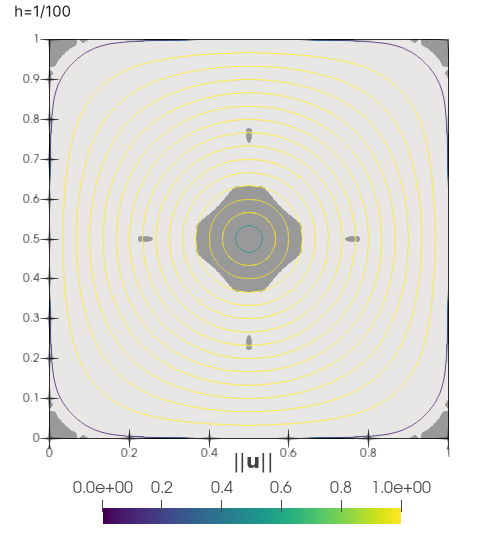}\\
			\includegraphics[width=0.4\textwidth]{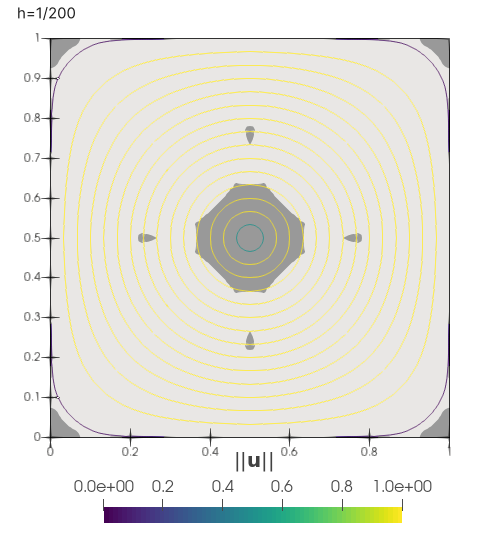}\includegraphics[width=0.4\textwidth]{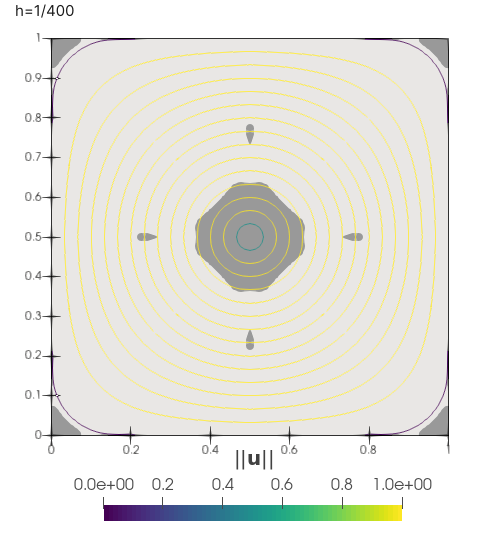}\\
			\includegraphics[width=0.45\textwidth]{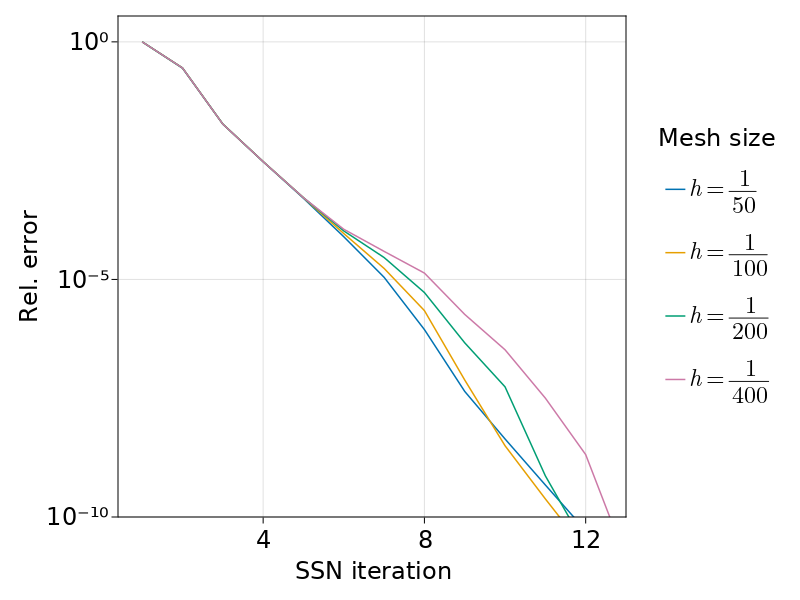}
		\end{center}
		\caption{\cblue{Reservoir Test. Top: velocity stream lines and Inactive set (dark gray). Bottom: Relative residual error and total number of iterations. Parameters: Herschel-Bulkley model $p=1.75$, tol=$\num{1e-10}$, $\gamma=\num{1e3}$, $\mu=1.0$, $\tau_s=10.0$.}} \label{img:resTest2_conv}
\end{figure}

	\begin{figure}[!t]
		\begin{center}
			\includegraphics[align=c,width=0.4\textwidth]{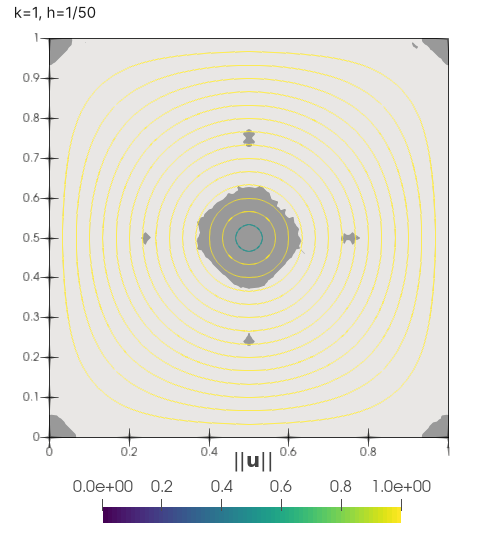}\includegraphics[align=c,width=0.45\textwidth]{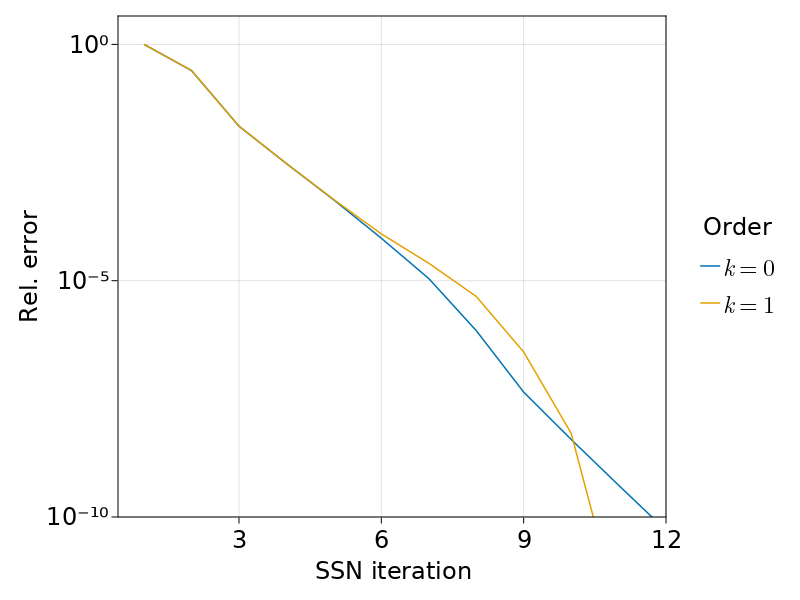}
		\end{center}
		\caption{\cblue{Reservoir Test. Top: velocity stream lines and Inactive set (dark gray) using $k=1$. Bottom: Relative residual error and total number of iterations for $k=0$ and $k=1$. Parameters: Herschel-Bulkley model $p=1.75$, $h=1/50$, tol=$\num{1e-10}$, $\gamma=\num{1e3}$, $\mu=1.0$, $\tau_s=10.0$.}}  \label{img:resTest2_conv2}
\end{figure}
\cblue{
	The influence of the mesh size has been investigated in Figure \ref{img:resTest2_conv}. It is shown that with respect to the number of iterations required to reach the prescribed tolerance ($\num{1e-10}$) the relation is quite weak, and the fast decay of the residuum is preserved through all the range of $h$ values studied. However, it's interesting to note, that even if the main features of the inactive zone are present and remain almost unchanged in all mesh sizes, some minor patters arise and develop only in finer meshes. Furthermore, in Figure \ref{img:resTest2_conv2}, we study the convergence with order $k=1$, as can be seen, increasing the order leads to faster error decay in the last iterations.}

\begin{figure}[!t]
	\begin{center}
		\includegraphics[width=0.55\textwidth]{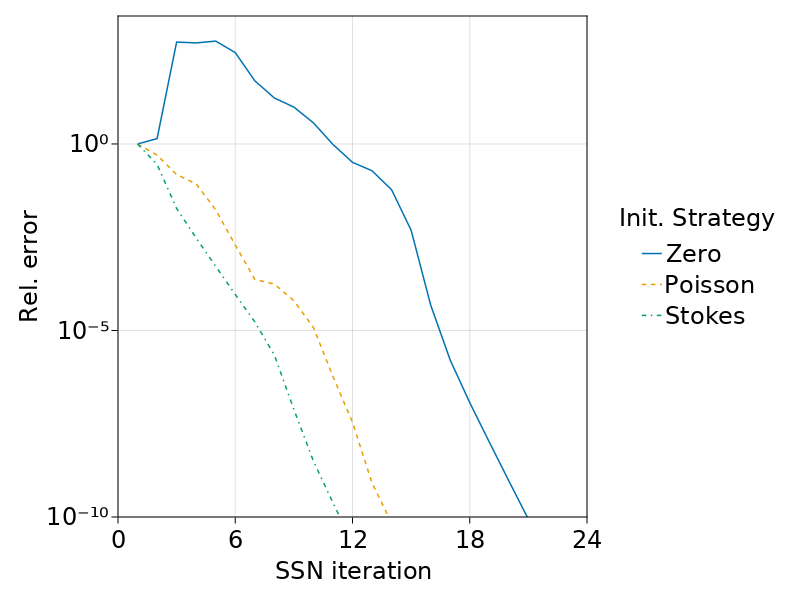}
	\end{center}
	\caption{\cblue{Reservoir Test. Relative residual error and total number of iterations under different initialization strategies. Parameters: Herschel-Bulkley model $p=1.75$, $h=1/100$, tol=$\num{1e-10}$, $\gamma=\num{1e3}$, $\mu=1.0$, $\tau_s=10.0$.}}  \label{img:resTest_convI}
\end{figure}
\cblue{Now we turn our attention to the influence of the initialization strategy on the convergence of the SSN iteration, results are shown in figure \ref{img:resTest_convI}. Even if a naive zero initialization for all variables still achieve convergence with fast error decay in the last steps, it is clearly better to start with a reasonable initial value for the velocity field (and in consequence for $\bs{\theta}$). In fact, we tested our scheme using the solution of a Poisson problem under the same boundary conditions and also solving a dual-mixed Stokes problem as alternative initialization strategies. We obtained error decay since the first iteration in both cases, with the Stokes strategy inducing the best results.}

	\begin{figure}[!t]
		\centering
		\begin{minipage}[b]{1.0\linewidth}
			\centering
			\includegraphics[align=c,width=0.5\textwidth]{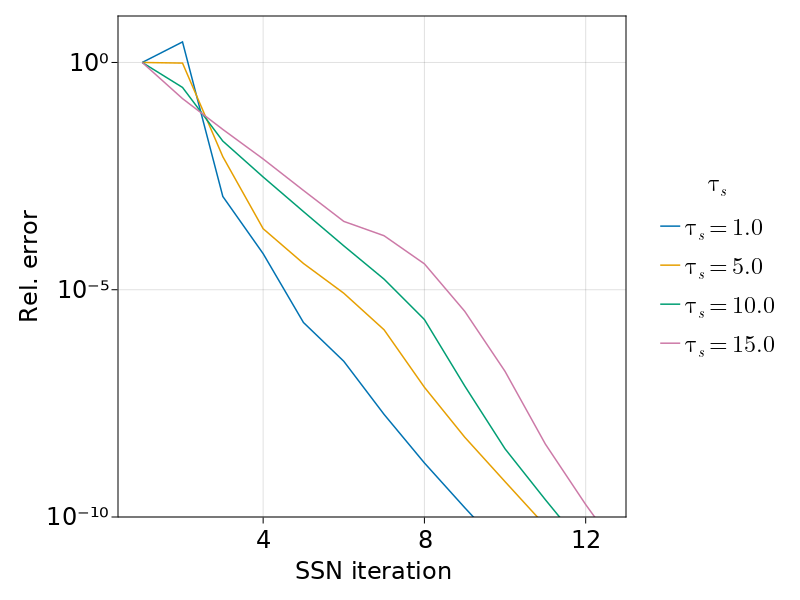}
			\qquad
			\begin{tabular}{r|cc}
				\hline 
				$\tau_s$ & \# Its & $|\mathcal{A}_{\gamma}|$ \tabularnewline
				\hline
				\hline
				1 & 9 &     39916 \tabularnewline
				5.0 & 10 &    39228 \tabularnewline
				10.0 & 11 &    37835  \tabularnewline
				15 & 12 &    30025  \tabularnewline
				\hline 
			\end{tabular}
			\caption{\cblue{Reservoir Test. For each value $\tau_s$: number of iterations and final size of the active set $\mathcal{A}_{\gamma}$. Parameters: Herschel-Bulkley model $p=1.75$, tol=$\num{1e-10}$, $k=0$, $h=1/100$, $\gamma=\num{1e3}$ and $\mu = 1.0$.}} \label{tab:resTaus}
		\end{minipage}
\end{figure}
\cblue{
	Finally, in Figure \ref{tab:resTaus}, we show the behavior of the numerical scheme for different values of $\tau_s$, with a focus on the numerical approximation of the final active set $\mathcal{A}_{\gamma}$. The SSN iteration behavior of the algorithm seems to increase slightly as we increase the value of $\tau_s$.}

\subsection{Flow in a driven cavity}

\begin{figure}[!t]
	\begin{center}
		\includegraphics[width=0.4\textwidth]{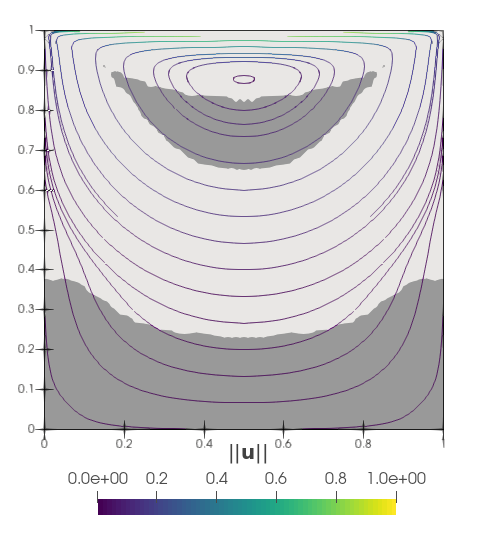}
		\includegraphics[width=0.4\textwidth]{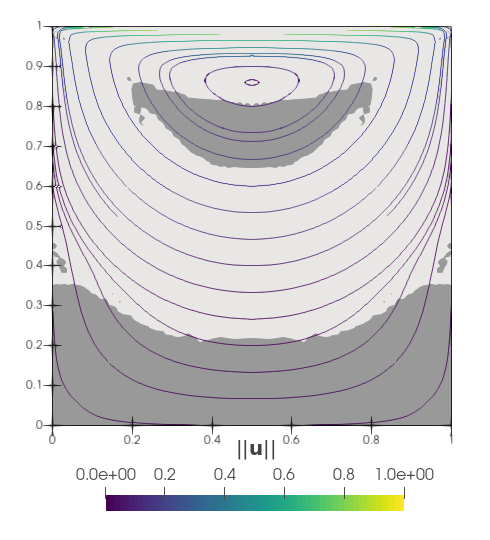}\\
		\includegraphics[width=0.4\textwidth]{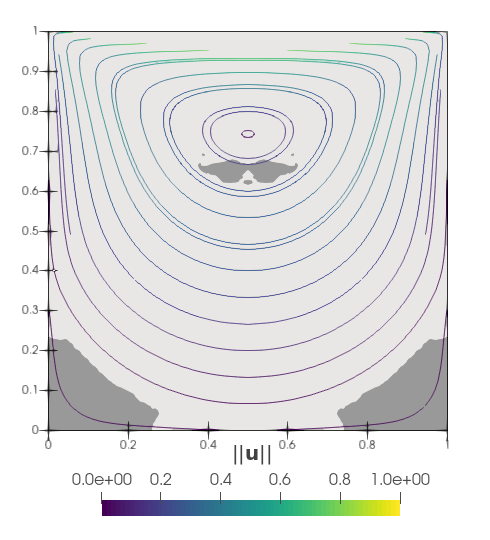}
	\end{center}
	\caption{Driven Cavity. velocity stream lines and \cblue{Inactive} set (dark gray) with Herschel-Bulkley model $p=1.6, 1.75 \text{ and } 4.0$ (from left to right and top to bottom). Parameters: $\gamma = \num{1e3}$, \cblue{$h=1/50$}, $\mu=1.0$.}  \label{img:cavTest1}
\end{figure}

Now, we simulate steady lid-driven cavity flow of a Herschel-Bulkley fluid in the unit square $\Omega_h \coloneqq (0,1) \times (0,1)$. We assume that $\Gamma_D = (x_1, 1)$, with $x_1 \in (0,1)$. We take $k=0$, $\mb{f} = 0$ and the following Dirichlet boundary condition
\begin{align*}
\mb{u}_h^D =
\begin{cases}
(1,0) &\text{if }x\in \Gamma_D \\
0 &		\text{otherwise}
\end{cases}
\end{align*}

Let us fix $\lambda = \num{1e3}$, $\mu = 1$ and $\tau_s = 2.5$. A first aim of this test is to assess how the method behaves for different values of $p$. The plots in Fig.~\ref{img:cavTest1} \cblue{computed with $h=1/50$ (equivalent to 10 000 cells)} show the different velocity vector field and flow streamlines patterns generated with $p = 1.6, 1.75$ and $4$. The results indicate that the formulation performs relatively well for this range of values.

Moreover, we performed the same tests with and without projection for $\mb{q}$ with similar results, suggesting that the modification introduced in \eqref{SSN4t} may be waived at the implementation level.

On other hand, Fig.~\ref{img:cavTest1} also displays the computed active and inactive sets, representing the rigid and plastic regions of the material. Here it is possible to observe the expected stagnation zones in the bottom of the cavity and the rigid zone in the upper part of the domain. Further, as various numerical simulations have shown concerning this test problem \cite{Glowinski2011}, the stagnation zones decrease for higher values of \cblue{$p$}.

\cblue{
\begin{figure}[!t]
	\begin{center}
		\includegraphics[align=c,width=0.45\textwidth]{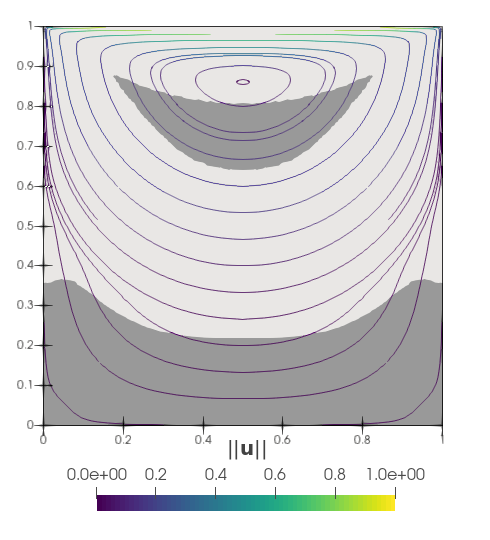}	\includegraphics[align=c,width=0.45\textwidth]{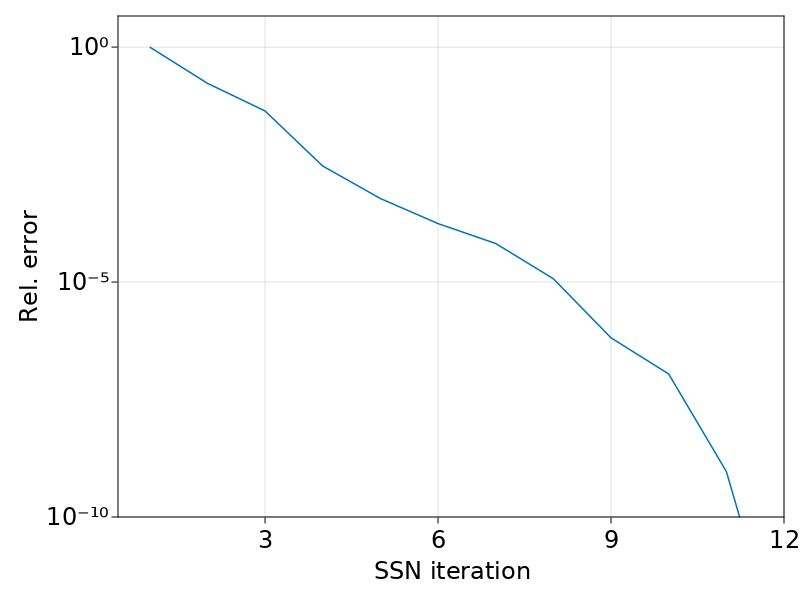}
	\end{center}
	\caption{Driven Cavity. Left: velocity stream lines and Inactive set (dark gray); Right: Relative residual error and total number of iterations. Parameters: \cblue{Herschel-Bulkley model $p=1.75$, tol=$\num{1e-10}$, h=1/100},  $\gamma=\num{1e3}$, $\mu=1.0$, $\tau_s=2.5$.}  \label{img:cavTest1_conv}
\end{figure}}
\begin{figure}[!t]
	\begin{center}
		\includegraphics[align=c,width=0.5\textwidth]{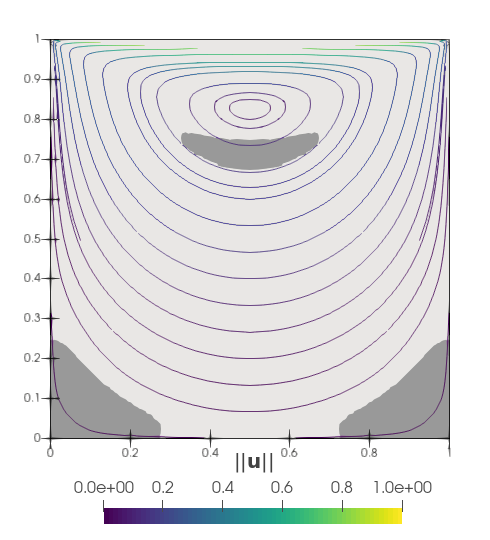}{\includegraphics[align=c,width=0.5\textwidth]{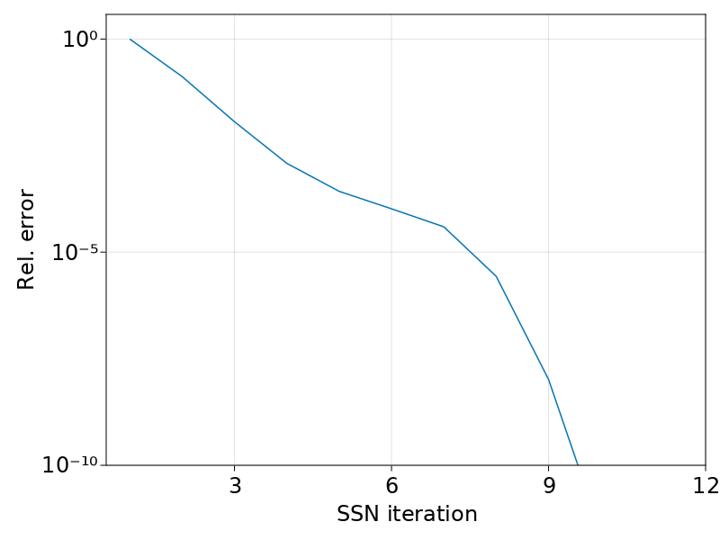}}
	\end{center}
	\caption{Driven Cavity. Left: \cblue{velocity stream lines and Inactive set (dark gray); Right: Relative residual error and total number of iterations}. Parameters: \cblue{Casson Law, tol=$\num{1e-10}$, h=1/100, k=0}, $\mu =1.0$, $\gamma = \num{1e3}$, $\tau_s=2.5$}  \label{img:cavTest1_casson}
\end{figure}
\cblue{
\begin{figure}[!t]
	\begin{center}
		\includegraphics[align=c,width=0.5\textwidth]{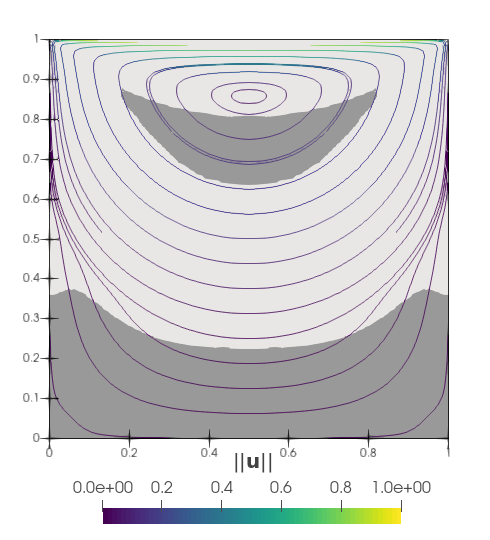}{\includegraphics[align=c,width=0.5\textwidth]{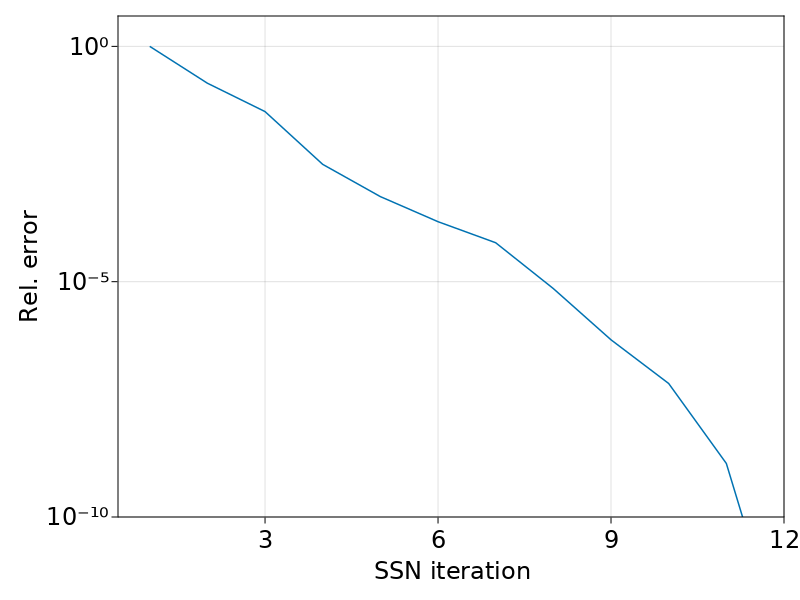}}
	\end{center}
	\caption{Driven Cavity. Left: \cblue{velocity stream lines and Inactive set (dark gray); Right: Relative residual error and total number of iterations}. Parameters: Carreau Fluid with yield $p=1.75$, tol=$\num{1e-10}$, h=1/100, k=0, $\mu =1.0$, $\gamma = \num{1e3}$, $\tau_s=2.5$}  \label{img:cavTest1_carreau}
\end{figure}}
The scheme is also tested with the \textit{Casson Law} model, qualitative results are displayed on figure \ref{img:cavTest1_casson}. Note that altough the vortex zone is wide as in \cblue{the} case of the Herschel-Bulkley model with $p=1.75$, the active zone \cblue{reaches} close to the bottom area similar to what is seen for the Herschel-Bulkley model with high $p$ values. \cblue{On the other hand, when we test the model with the Carreau law with yield, we got results close to those seen with the Herschel-Bulkley model.}

Further, on the right side of figures \ref{img:cavTest1_conv}\cblue{, \ref{img:cavTest1_casson} and \ref{img:cavTest1_carreau}}, we show the values of \cblue{the relative residual error}, as well as the number of Newton iterations. The fast decay of the
residuum in the last iterations, which illustrates the local superlinear convergence rate of the algorithm is noticeable \cblue{in all cases.}

\subsection{Flow in a bounded channel}

\begin{figure}[!t]
	\begin{center}
		\includegraphics[align = c,width=0.5\textwidth]{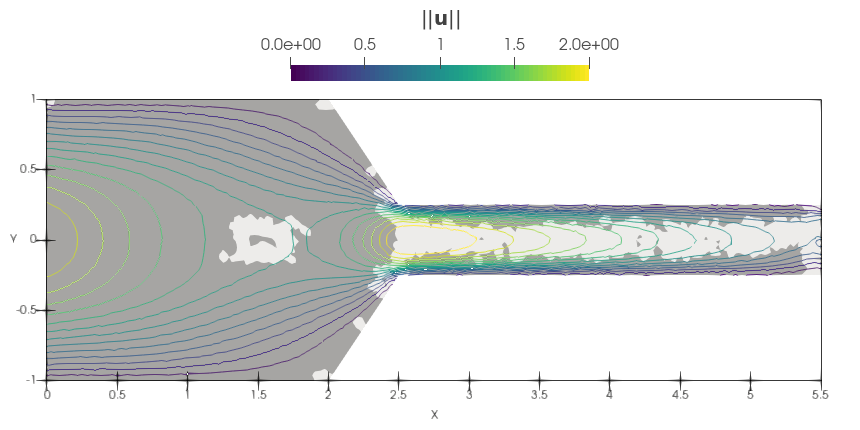}\includegraphics[align = c,width=0.5\textwidth]{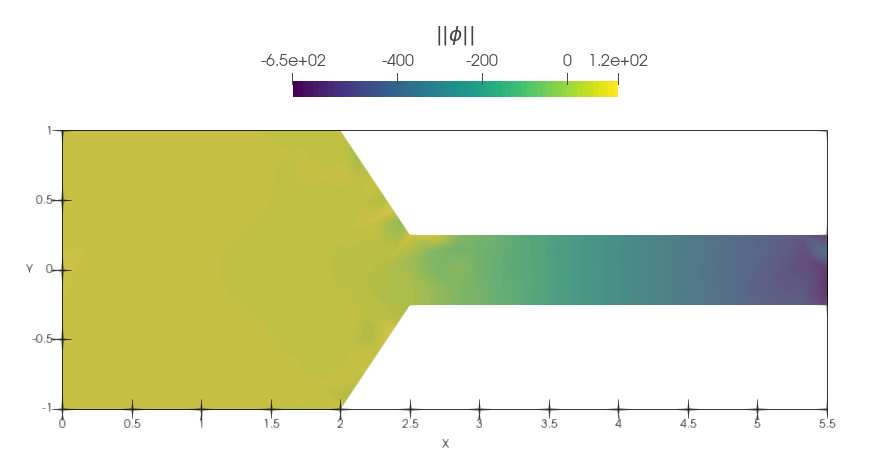}\\
		\includegraphics[align = c,width=0.5\textwidth]{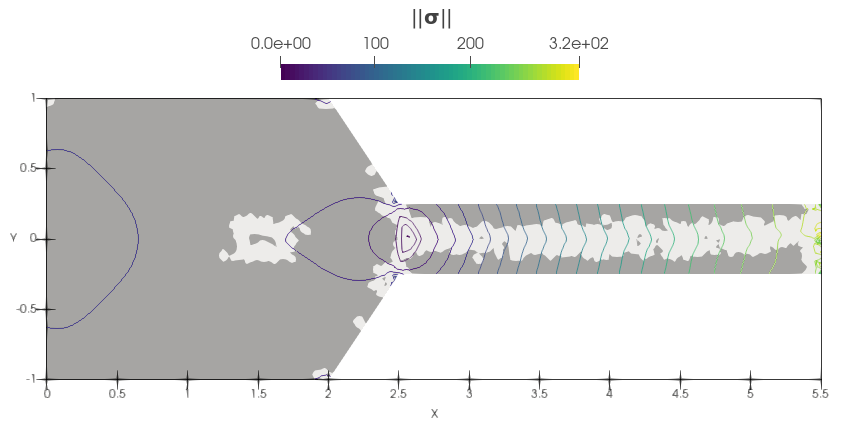}\includegraphics[align = c,width=0.5\textwidth]{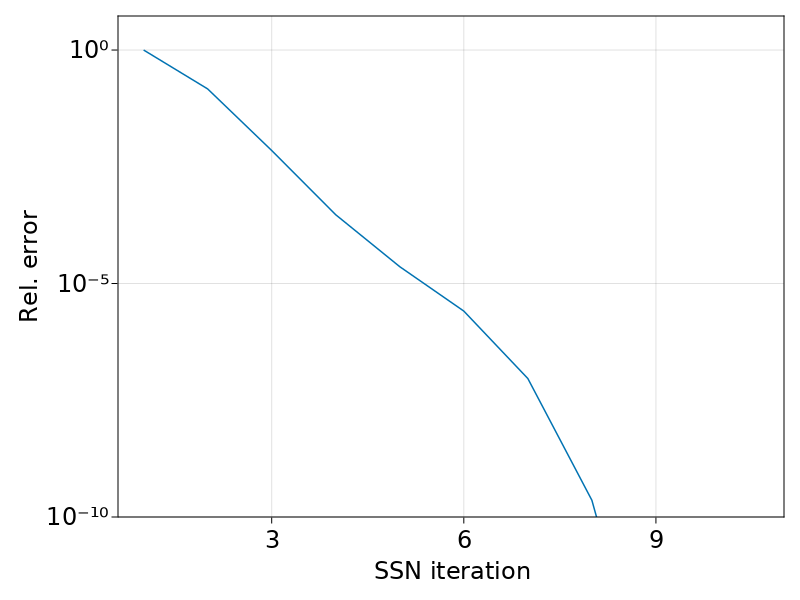}
	\end{center}
	\caption{Bounded Channel. Top Left: \cblue{velocity stream lines, Top right: pressure $\phi$, Bottom Left: $\bs{\sigma}$ tensor magnitude contours, Bottom right: Relative residual error and total number of iterations. (Inactive set in dark gray).
		Parameters: Casson Law, $k=0$, tol = $\num{1e-10}$}, $\gamma = \num{1e3}, \tau_s=2.5$.}  \label{img:boundedChannel2d}
\end{figure}

In this example we compute the flow of a fluid following the Casson Law, through a bounded channel with a reduction in the section width, \cblue{using $k=0$ and 10200 cells}. For the velocity we set a Dirichlet boundary condition on the inflow (left wall): $\mb{u}_D=(1-y^2,0)$, not penetrating boundary conditions on the walls of the channel and a stress-free boundary condition for the outflow (right wall). We also assume that there is not a forcing term, i.e., $\mb{f} = (0, 0)$ , and the material is expected to flow just under the effect of the inflow $\mb{u}_D$.

\cblue{In Figure \ref{img:boundedChannel2d} velocity vector streamlines, pressure $\phi$,  magnitude of the tensor $\bs{\sigma}$ contours, final active and inactive sets and residual errors, are shown. While pressure gradient follows the expected pattern, decreasing in the direction of flow.} Results indicate that increasing the magnitude of fluid velocity has a decreasing effect on yield stress. Accordingly, downstream of the contraction we find the active set with the highest velocity magnitude, while the inlet of low speed contains the inactive set. \cblue{In this particular example, some numerical artifacts are also distinguishable near the non-slip boundaries. It can be related with the combination of Dirichlet and the so called do nothing condition at the outflow. Indeed, a more in deep study of the interaction of our formulation with other boundary conditions is an interesting direction for future work}. 

\subsection{Reservoir flow in three dimensions}

\begin{figure}[!t]
	\begin{center}
		\includegraphics[align=c,width=0.4\textwidth]{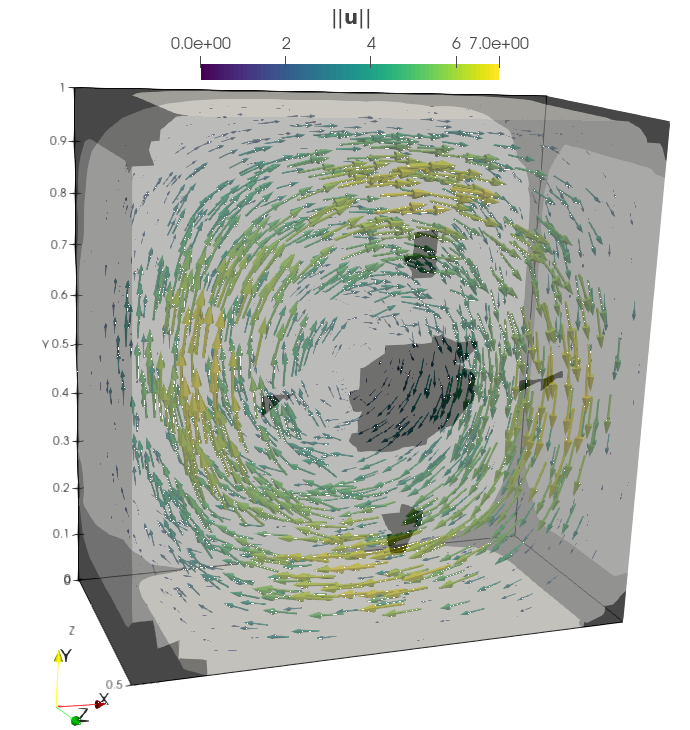}\includegraphics[align=c,width=0.4\textwidth]{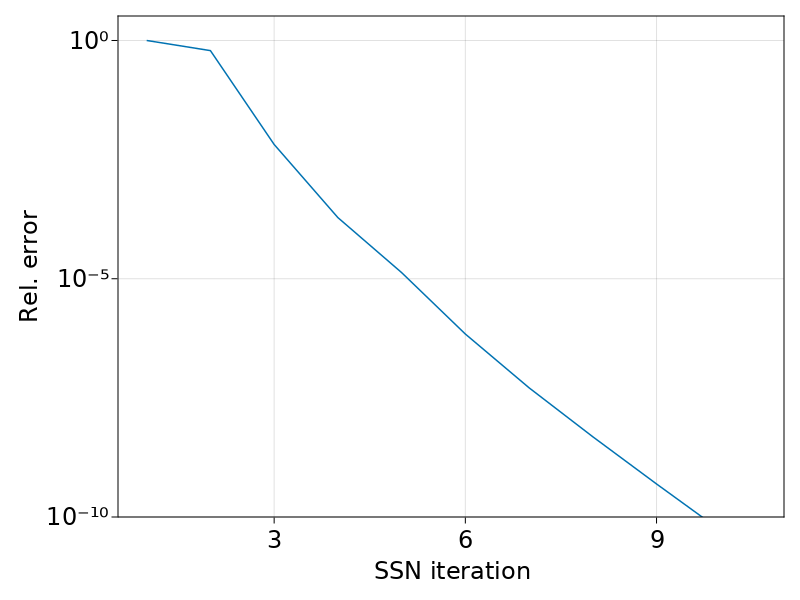}
	\end{center}
	\caption{\cblue{3D Reservoir flow. Left: velocity stream lines and Inactive set (Dark gray); Right: Relative residual error and total number of iterations.} Parameters: \cblue{Herschel-Bulkley model $p =1.75$, tol=$\num{1e-10}$, k=0}, $\gamma = \num{1e3}$, $\tau_s=5.0$.}  \label{img:resTest3}
\end{figure}

\cblue{Now, we present a simple extension of our first test to the three dimensional case.} We choose a unit cube $(0,1)\times(0,1)\times(0,1)$ as domain and the following parameters: $p=1.75$, $\tau_s = 5.0$, $\mu = 1$, \cblue{$k=0$, } and $\mb{f}(x_1,x_2,x_3) \coloneqq 300(x_2 - 0.5, 0.5 - x_1,0.0)$. The discretization of $\Omega$ is done with tetrahedrons, \cblue{using 66000 cells}. \cblue{Figure \ref{img:resTest3} shows a clipped cross section with the streamlines of fluid and particle velocities $\mb{u}$ on the left panel. The fluid velocity streamlines indicate the direction of the flow and the generation of axisymmetric recirculation patterns. Also, in the same figure the computed active and inactive sets are displayed.  As expected, the results of this three-dimensional case largely coincide with the results of the two-dimensional setting. Figure \ref{img:resTest3} right, shows convergence of the SSN iteration is slightly slower than the two dimensional counterpart.}

\subsection{Cavity Test flow in three dimensions}

\begin{figure}
	\begin{center}
		\includegraphics[align=c,width=0.5\textwidth]{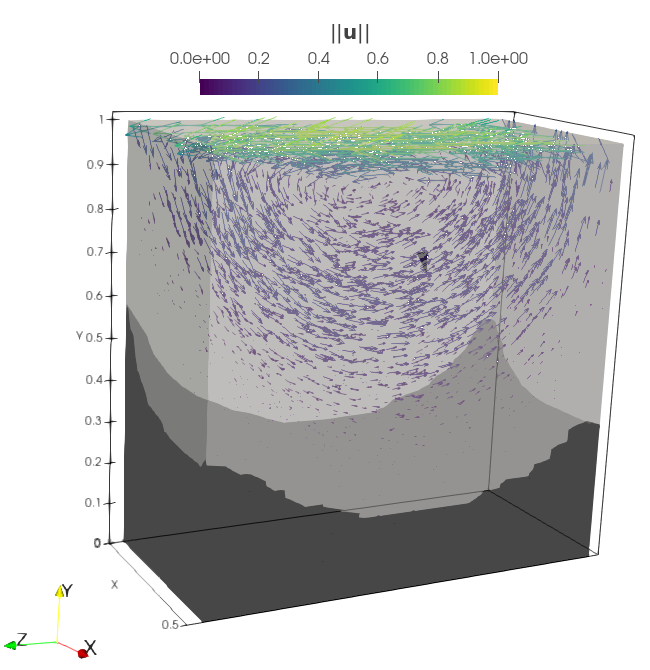}\includegraphics[align=c,width=0.5\textwidth]{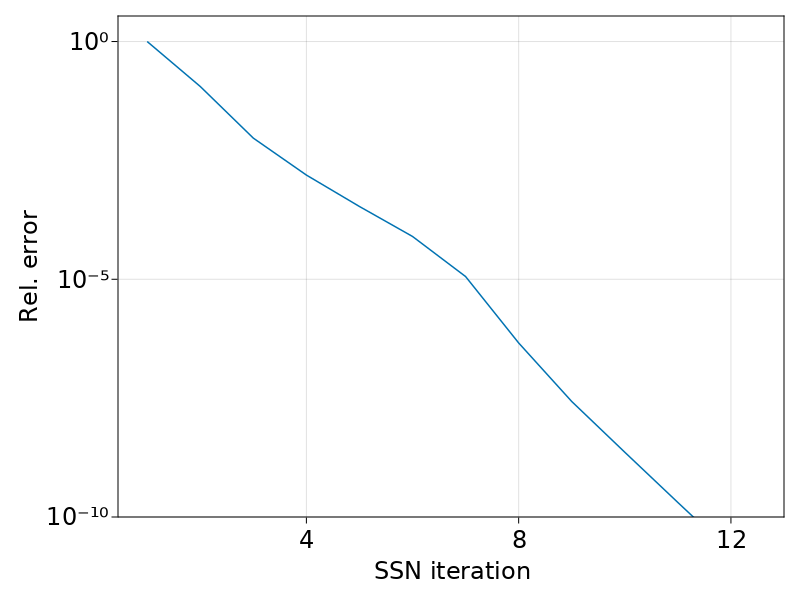}
	\end{center}
	\caption{\cblue{3D Cavity Test flow. Left: velocity stream lines and Inactive set (dark gray); Right: Relative residual error and total number of iterations.  Parameters: Herschel-Bulkley model $p=1.75$, tol = $\num{1e-10}$, k=0, $\gamma = \num{1e3}$, $\mu=1.0$, $\tau_s=2.5$.}}  \label{img:resTest4}
\end{figure}

Finally, we study a driven cavity flow in three dimensions \cblue{with the Herschel-Bulkley model}. Incompressible flows inside a 3D lid-driven cavity are of great interest because
they exhibit challenging fluid mechanical phenomena, such as instabilities, transition, turbulence, and complex 3D flow patterns. The domain is a unitary cube $\Omega = [0,1]\times[0,1]\times[0,1]$, with boundary condition:
\begin{align*}
	\mb{u}_h^D =
	\begin{cases}
		(0,0,1) &\text{if }x\in \Gamma_D := \{1\}\times[0,1]\times[0,1] \\
		(0,0,0) &		\text{otherwise}
	\end{cases}
\end{align*}
\cblue{We discretized the computational domain by a tetrahedral grid with 63900 cells and $k=0$, other relevant parameters are chosen as follows: $\gamma = \num{1e3}$, $\mu=1.0$, $\tau_s=2.5$.
Streamline plots of the computed velocity field together with the computed active and inactive sets are plotted in Figure \ref{img:resTest4} left. Residual error and total SSN iteratons are also displayed on Figure \ref{img:resTest4}. As in the previous example the convergence is slightly slower than the two dimensional case.}

\section{Discussion and concluding remarks}\label{sec:concl}
We have advanced a numerical scheme for the approximation of a general class model for viscoplastic fluids with yield, our approach incorporates a Huber type regularization and a \textit{dual-dual} mixed formulation. 

The numerical method we have used is based on Arnold-Falk-Winther finite elements and discontinuous finite elements for the flow. A semismooth Newton scheme with exact Jacobian has 
been employed for the linearization of the resulting system.

We have generated several numerical tests to verify if the proposed algorithm leads to numerical solutions that preserve important physical properties of the flow. Furthermore, we studied a wide range of models for viscoplastic materials with yield, including Herschel-Bulkley, Casson and Carreau with yield models. \cblue{The qualitative results displayed in section \ref{sec:results} are encouraging, moreover, because of the active–inactive set structure of the proposed numerical scheme, we obtain a reliable and accurate enough prediction of the final solid–fluid zones of the material in two and three-dimensional settings.}

\bigskip
\noindent\textbf{Acknowledgement.} We acknowledge the partial support by Escuela Polit\'ecnica Nacional del Ecuador, under the projects PIS 18-03 and PIGR 19-02. This research was carried out using the research computing facilities offered by Scientific Computing Laboratory of the Research Center on Mathematical Modeling: MODEMAT, Escuela Politécnica Nacional - Quito. Finally, we would like to thank the anonymous referees for many helpful comments which lead to a significant improvement of the article.



\end{document}